\documentclass[12pt,reqno]{amsart}               %%   `reqno' added to documentclass by DD to put equation numbers on right 

\calclayout

\IfFileExists{mathrsfs.sty}{\usepackage{mathrsfs}}
{\IfFileExists{euscript.sty}{\usepackage[mathscr]{euscript}}
{\let\mathscr\mathcal}}

\usepackage{amsfonts}
\usepackage{amssymb}
\usepackage{amscd}
\usepackage{amsthm}
\usepackage{url}

\input cyracc.def
\newfam\cyrfam
\font\tencyr=wncyr10
\font\sevencyr=wncyr7
\font\fivecyr=wncyr5

\textfont\cyrfam=\tencyr \scriptfont\cyrfam=\sevencyr
  \scriptscriptfont\cyrfam=\fivecyr
\makeatletter
\def\@tocline#1#2#3#4#5#6#7{\relax
  \ifnum #1>\c@tocdepth % then omit
  \else
    \par \addpenalty\@secpenalty\addvspace{#2}%
    \begingroup \hyphenpenalty\@M
    \@ifempty{#4}{%
      \@tempdima\csname r@tocindent\number#1\endcsname\relax
    }{%
      \@tempdima#4\relax
    }%
    \parindent\z@ \leftskip#3\relax \advance\leftskip\@tempdima\relax
    \rightskip\@pnumwidth plus4em \parfillskip-\@pnumwidth
    #5\leavevmode\hskip-\@tempdima
      \ifcase #1
       \or\or \hskip 2em \or \hskip 2em \else \hskip 3em \fi%
      #6\nobreak\relax
\hfill\hbox to\@pnumwidth{\@tocpagenum{#7}}\par
    \nobreak
    \endgroup
  \fi}
\makeatother

 \usepackage[pdfpagelabels]{hyperref}
 \hypersetup{pdftitle={Decomposition Types in Minimally Tamely Ramified Extensions of Q},pdfauthor={D.\ Dummit and H.\ Kisilevsky}} 
 \hypersetup{colorlinks=true,linkcolor=blue,anchorcolor=blue,citecolor=blue}

\numberwithin{equation}{section}              % added by DD to include section in equation numbers

\newcommand{\Z}{\mathbb{Z}}
\newcommand{\ZZ}{\mathbb{Z}}                              % added by DD to conform with past usage

\newcommand{\Q}{\mathbb{Q}}
\newcommand{\QQ}{\mathbb{Q}}                              % added by DD to conform with past usage

\newcommand{\FF}{\mathbb{F}}                             % added by DD 

\newcommand{\Gal}{\operatorname{Gal}}

\newcommand{\rank}{\operatorname{rank}}

\newcommand{\isom}{\simeq}
\newcommand{\ph}{\varphi}

\newcommand{\Fr}{\mathrm{Fr}}                        %    added by D^2   5-4-2016
\newcommand{\defi}[1]{\textsf{#1}}                   %    added by D^2   5-4-2016 

\newtheorem{thm}{Theorem}[section] 
\newtheorem{prop}[thm]{Proposition}
\newtheorem{lem}[thm]{Lemma}
\newtheorem{cor}[thm]{Corollary}
\theoremstyle{definition}

\newtheorem{defn}[thm]{Definition}

\theoremstyle{remark}                                %    added by D^2   5-4-2016
\newtheorem{remark}[thm]{Remark}                     %    added by D^2   5-4-2016
\usepackage{mathtools}                                                            %    added by D^2   5-7-2016
\def\gp#1{\langle \, #1 \, \rangle}                                               %    added by D^2   5-7-2016

\def\order#1{\vert \, #1 \, \vert}

\title[ Decomposition Types in Minimally Tamely Ramified Extensions
of $\Q$]%
{ Decomposition Configuration Types in Minimally Tamely Ramified Extensions
of $\Q$}

\thanks{${}^{\dagger} $This work was supported in part by a grant from NSERC}
\date{}
\author[D.~Dummit and H.~Kisilevsky]{David S. Dummit and Hershy Kisilevsky$^{\dagger}$}
\address[H.~Kisilevsky]%
{Department of Mathematics and Statistics and CICMA\\
     Concordia University \\
     1455 de Maisonneuve  Blvd. West\\
     Montr\'eal, Quebec, H3G 1M8, CANADA}
\email{hershy.kisilevsky@concordia.ca}
\address[D.~Dummit]%
{Department of Mathematics \\ University of Vermont \\ Lord House, 16 Colchester Ave. \\ Burlington, VT 05405, USA}
\email{dummit@math.uvm.edu}


\begin{document}

\date{\today}

\begin{abstract} 

We examine whether it is possible to realize finite groups $G$ as Galois groups of minimally tamely
ramified extensions of $\QQ$ and also specify both the inertia groups and 
the further decomposition of the ramified primes.

\end{abstract}

\maketitle

\tableofcontents

\section{Introduction}\label{Intro}

Let $G$ be a finite group and let $s$, the rank of $G$ ($=\rank(G)$)  denote the minimal number of elements required to {\it normally} 
generate $G$, i.e., $s$ is the minimal number of elements of $G$ which together with all their conjugates in $G$ generate $G.$  
It is known  (\cite{Ku}) that $s$ is the minimal number of generators of the maximal abelian quotient $G/[G,G]$ of $G.$

If $K$ is a finite Galois extension of $\QQ$, then because 
$\QQ$ has no unramified extensions, $\Gal(K/\QQ)$ is generated by the inertia groups for the primes $p$ ramifying in $K$.  
If all of the ramification in $K$ is tame then these inertia groups are cyclic, and taking a generator for one fixed representative
of the (conjugate) inertia groups for each prime $p$ gives a set of normal generators for $\Gal(K/\QQ)$.  It follows that if
a finite group $G$ can be realized as the Galois group of a number field $K$ over $\QQ$ 
having only tame ramification, then $\rank(G)$ is the smallest possible number of primes $p$ that are ramified in $K$.  
In \cite{KS} and \cite{KNS} 
it is shown that all finite nilpotent semi-abelian groups can be realized by such a minimally tamely 
ramified extension over $\Q.$

In this paper we consider the finer question of whether it is possible to realize finite groups $G$ as Galois groups of minimally tamely
ramified extensions of $\QQ$ and also specify both the cyclic subgroups of $G$ arising as the inertia groups and 
the further decomposition of the ramified primes in $K$.

\section{Minimal Tame Ramification and Decomposition Configurations}\label{RamDecomp}

We first make precise what is meant by a specification of the inertia and decomposition groups in a realization
of a finite group $G$ as a Galois group over $\QQ$.

\begin{defn} \label{def:tameramconf}

A (minimal) \defi{tame ramification configuration} 
is a pair $(G,\mathcal T)$, where $G$ is a finite group of rank $s$ and  $\mathcal T$ is a 
(necessarily minimal)
collection 
$\mathcal T= \{T_1, \dots, T_s\}$ of cyclic subgroups of $G$ that normally generate $G$.

\end{defn}

\begin{defn} \label{def:realizedtameramconf}

A tame ramification configuration is \defi{realizable over $\QQ$} if there exists a tamely 
ramified Galois extension $K/\QQ$ and an isomorphism $\ph:\Gal(K/\QQ) \longrightarrow G$ such that  
$\ph(T(\wp_i/(p_i))) =T_i$ for each $i=1, \dots s$, where 
$\{ p_1, \dots, p_s \}$ is the set of all finite primes of $\QQ$ ramified in $K$ and for each $i$, $\wp_i$ is a prime 
of $K$ dividing $(p_i)$ with inertia group $T(\wp_i/(p_i))$ in $\Gal(K/\QQ)$.

\end{defn}

An extension $K/\QQ$ is minimally tamely ramified if and only if it is tamely ramified and is the 
realization of a minimal tame ramification configuration.

With the evident minor modifications one could consider tame ramification configurations that are not minimal (where
the cardinality $n$ of $\mathcal T$ need not be the same as the rank of $G$ and the $T_i$ need not normally generate $G$),
and realizations over number fields $F$ other than $\QQ$.
So, for example, if $\mathcal T = \emptyset$, then $(G, \mathcal T)$ would be realizable over $F$ if and only if 
$F$ has an unramified $G$-extension.  If $F$ has no unramified extensions, as is the case for $\QQ$,
then for a tame ramification configuration $(G, \mathcal T)$ to be realizable, 
the groups $T_i$ must normally generate $G$, and hence  $n \geq s$. 
Our definitions above reflect the fact that we shall only consider minimally tamely ramified
extensions over $\QQ$ in this paper.

\medskip\indent
If $K/\QQ$ is a Galois extension of number fields and $\wp$ is a prime of $K$ lying above the prime $(p)$,
then the decomposition group $Z(\wp/(p))$ for $\wp$ in $\Gal(K/\QQ)$ contains the inertia group $T(\wp/(p))$
as a normal subgroup and the quotient is a cyclic group, which leads to the following definitions.

\begin{defn}  \label{def:tamedecconf}
A (minimal) \defi{tame decomposition configuration} is a triple $(G,\mathcal T, \mathcal Z)$, where $(G, \mathcal T)$ is a 
tame ramification configuration, and
$\mathcal Z$ is a collection $\mathcal Z= \{Z_1, \dots, Z_s\}$ of subgroups of $G$ where $T_i$ is a normal subgroup 
of $Z_i$ and $ Z_i/T_i$ is cyclic.   
\end{defn}

\begin{defn}  \label{def:realizedtamedecconf}
A (minimal) tame decomposition configuration is \defi{realizable over $\QQ$} if there exists a tamely 
ramified Galois extension $K/\QQ$ and an isomorphism $\ph:\Gal(K/\QQ) \longrightarrow G$ such that  
$\ph(T(\wp_i/(p_i))) =T_i$ as in Definition \ref{def:realizedtameramconf}
and $\ph(Z(\wp_i/(p_i))) =Z_i$ for each $i=1, \dots s$, where $Z(\wp_i/(p_i))$ is the
decomposition group for $\wp_i$ in $\Gal(K/\QQ)$.

\end{defn}
Recall that if $K/\QQ$ is a Galois extension and $K_0$ is a subfield of $K$ 
Galois over $\QQ$, then
the inertia and decomposition groups in $\Gal(K/\QQ)$ 
project to the inertia and decomposition groups in $\Gal(K_0/\QQ)$ for the corresponding primes of $K_0$.

\begin{defn} \label{def:quotient}
Call a tame ramification configuration $(H, \mathcal S)$ a \defi{quotient of a tame ramification configuration} $(G, \mathcal T)$  
if $\rank(H)=\rank(G)$ and there is a surjective group homomorphism $\pi:G \longrightarrow H$ such that 
$\pi(T_i)=S_i$ for all $T_i \in \mathcal T$ and all $S_i \in \mathcal S$. 
Similarly, a tame decomposition configuration $(H, \mathcal S, \mathcal W)$ is a \defi{quotient of a tame decomposition 
configuration} $(G, \mathcal T, \mathcal Z)$ if $(H, \mathcal S)$ a quotient of  
$(G, \mathcal T)$ and in addition $\pi(Z_i)=W_i$ for all $Z_i \in \mathcal Z$ and 
$W_i \in \mathcal W$. 
\end{defn}

If $K/\QQ$ is a realization of the tame ramification configuration $(G, \mathcal T)$ 
(respectively, of the tame decomposition configuration $(G, \mathcal T, \mathcal Z)$), 
then $K_0/\QQ$ will be a realization of the tame ramification configuration $(H, \mathcal S)$ (respectively, 
of the tame decomposition configuration $(H, \mathcal S, \mathcal W)$),
where $K_0$ is the subfield of $K$ fixed by subgroup of $\Gal(K/\QQ)$ corresponding (under $\ph$) 
to the kernel of $\pi$, which proves the following.

\begin{prop} \label{prop:quotient}
Suppose $(H, \mathcal S)$ (respectively, $(H, \mathcal S, \mathcal W)$) is a quotient of some tame 
ramification configuration $(G, \mathcal T)$ (respectively, tame decomposition configuration 
$(G, \mathcal T, \mathcal Z)$).  Then 
\begin{enumerate}\itemsep0pt
\item[{(a)}]
if $(H, \mathcal S)$ (resp., $(H, \mathcal S, \mathcal W)$) cannot be realized over $\QQ$, then
neither can $(G, \mathcal T)$ (resp., $(G, \mathcal T, \mathcal Z)$), and
\item[{(b)}]
if $(G, \mathcal T)$ (resp., $(G, \mathcal T, \mathcal Z)$) can be realized over $\QQ$, then
so can $(H, \mathcal S)$ (resp., $(H, \mathcal S, \mathcal W)$)).
\end{enumerate}
\end{prop}

\section{Decomposition Configurations in Finite Abelian Groups}   \label{Abelian}

The rank $s$ of a finite abelian group $G$ is the number of cyclic factors in its invariant factor decomposition,
and a tame ramification configuration is simply a specification of $s$ generators, $x_1, \dots , x_s$ of $G$, the minimum possible. 
Then a tame decomposition configuration is an additional choice of elements $z_1, \dots , z_s \in G$.   

We first use Proposition \ref{prop:quotient} to show that tame decomposition configurations can be realized for 
abelian groups if they can be realized for the particular abelian groups $(\ZZ/n\ZZ)^s$.

\begin{prop} \label{lift}
Let $H$ be a finite abelian group with $\rank(H)=s$, and let $(H,\mathcal S, \mathcal W)$ be a tame 
decomposition configuration.   Then there is a tame decomposition configuration $(G,\mathcal T, \mathcal Z)$ 
with $G = (\Z/n\Z)^s$ that has $(H,\mathcal S, \mathcal W)$ as a quotient.   In particular, if
$(H,\mathcal S)$ is a tame ramification configuration, then 
there is a tame ramification configuration $(G,\mathcal T)$ with $G = (\Z/n\Z)^s$ which has $(H,\mathcal S)$ as a quotient.

\end{prop}

\begin{proof}
  
For each $i=1,\dots,s$ let $x_i \in S_i$ be a generator of the cyclic group $S_i \in \mathcal S$ and suppose 
the decomposition group $W_i \in \mathcal W$ is generated by $\{ x_i, z_i \}$. 
Let $\psi:\Z^s \longrightarrow H$ 
be the group homomorphism defined by  $\psi(\epsilon_i)=x_i$, where $\epsilon_i=(0,\dots,0,1,0,\dots,0)$ with $1$ in the 
$i^{\rm{th}}$ position. For each $i = 1, \dots, s$, choose $\eta_i \in \Z^s$ so that $\psi(\eta_i)=z_i$.
Fix any positive integer $n$ divisible by the exponent of $H$,
so that $n\cdot x=0$ for all $x \in H$, and let
$f:\Z^s \longrightarrow G = \Z^s/n\Z^s$ be the natural projection.  
Then $n\Z^s$ is contained in the kernel of $\psi$, 
so $\psi$ factors through $f$ and induces a surjective homomorphism  $\pi=\overline\psi:G  \longrightarrow H$. 
For $i = 1, \dots, s$, let  $T_i = \langle f(\epsilon_i) \rangle $ be the subgroup of $G$
generated by  $ f(\epsilon_i)$ and let $Z_i = \langle f(\epsilon_i),f(\eta_i) \rangle$ 
be the subgroup of $G$ generated by  $ f(\epsilon_i)$ and $f(\eta_i)$.
If $\mathcal T = \{ T_1, \dots , T_s \}$ and  $\mathcal Z = \{ Z_1, \dots , Z_s \}$, then
$(G,\mathcal T, \mathcal Z)$ is a tame decomposition configuration with 
$\rank G = s = \rank H$ that has 
$(H,\mathcal S, \mathcal W)$ as a quotient.
\end{proof}
 
It follows from Propositions \ref{prop:quotient} and \ref{lift} that 
to construct a realization for the tame ramification configuration $(H,\mathcal S)$ or for the 
tame decomposition configuration $(H,\mathcal S, \mathcal W)$ for an abelian group
$H$, it would suffice if we could
construct a realization for all possible configurations for the group $G = (\Z/n\Z)^s$.
We shall see that this can be done for {\it ramification\/} configurations (see
Theorem \ref{thm:split})
but not in general for all possible {\it decomposition\/} configurations.

If $G = (\Z/n\Z)^s$, each $T_i \in \mathcal T$ has order at most $n$, so the fact that $\rank G = s$ and
$G$ is generated by the $T_i$ implies that $T_i \simeq Z/n\Z$ for all $i = 1, \dots, s$
and that $G = T_1 \times \dots \times T_s$ in Proposition \ref{lift} (so, up to an evident equivalence, there
is essentially only one tame ramification configuration for $(\Z/n\Z)^s$).
Then if $K/\QQ$ is a realization of $(G, \mathcal T)$,
taking the fixed fields of the subgroups $T_1 \times \dots \times T_{i-1} \times T_{i+1} \times \dots T_s$
for $i = 1, \dots s$ shows that $K$ would be the composite of cyclic extensions of degree $n$, 
each of which is totally and tamely ramified at one odd prime and otherwise unramified at finite primes.  
These extensions are the cyclic subextensions of prime cyclotomic fields:

\begin{defn}  \label{def:Kp}
If $n$ is a positive integer and $p$ is a prime with $p \equiv 1$ mod $n$, let $K_n(p)$ 
denote the subfield of degree $n$ contained in the cyclotomic field 
of $p^\textup{th}$ roots of unity.
\end{defn}

If $\zeta_p$ is any primitive $p^\textup{th}$ root of unity, the Galois group of $\QQ(\zeta_p)/\QQ$ is
canonically isomorphic to $(\ZZ /  p\ZZ)^{\times}$ under the map sending $a \in \ZZ$ to the automorphism
$\sigma_a : \zeta_p \mapsto \zeta_p^a$.  The Galois group $\Gal(K_n(p)/\QQ)$ of the unique subfield of 
degree $n$ in $\QQ(\zeta_p)$ is then 
isomorphic to $(\ZZ /  p\ZZ)^{\times} / (\ZZ /  p\ZZ)^{\times n}$.  This latter group is (non-canonically) isomorphic
to $\ZZ / n \ZZ$, with an isomorphism obtained by choosing a generator for the cyclic group $(\ZZ /  p\ZZ)^{\times}$, i.e., by
choosing a primitive root $g$ modulo $p$.  Then
\begin{equation} \label{eq:Kpisom}
\Gal(K_n(p)/\QQ) = \langle \tau_g \rangle ,
\end{equation}
where $\tau_g$ is the restriction to $K_n(p)$ of the automorphism $\sigma_g \in \Gal(\QQ(\zeta_p)/\QQ)$.  
In particular, if $l$ is a prime distinct from $p$, then 
since $\sigma_l$ is the Frobenius automorphism for $(l)$ in $\Gal(\QQ(\zeta_p)/\QQ)$, whose
restriction to $K_n(p)$ is the Frobenius automorphism for $(l)$ in $\Gal(K_n(p)/\QQ)$, we see that
\begin{equation}  \label{eq: frobinKp} 
\Fr_{K_n(p)/\QQ} (l) = \tau_g^{b} \in \Gal(K_n(p)/\QQ) \quad \text{if} \quad l \equiv g^b \mod p .
\end{equation}

By Dirichlet's Theorem, there exist distinct primes $l_1, \dots , l_s$ each of which is congruent to 1 mod $n$.
By the remarks above, the composite $K = K_n(l_1) \cdots K_n(l_s)$
is then a realization over $\QQ$ of the tame ramification
configuration $(G,\mathcal T)$ with $G=(\Z/n\Z)^s$ and $T_i = \Z/n\Z$ for $i=1,\dots,s$.
With little additional effort we can do slightly more.  We first 
give a name to the decomposition configurations corresponding to the situation
in which the ramifying primes are otherwise totally split:

\begin{defn} \label{def:splitdec}
A tame decomposition configuration $(G, \mathcal T, \mathcal Z)$ is a (minimal) \defi{tame split decomposition 
configuration} if $T_i=Z_i$ for $1 \leq i \leq s.$  
\end{defn}

The field $K = K_n(l_1) \cdots K_n(l_s)$ considered above will be
a realization over $\QQ$ of the tame split decomposition
configuration $(G,\mathcal T, \mathcal Z)$ where $G=(\Z/n\Z)^s$ and $Z_i=T_i =\Z/n\Z$ for $i=1,\dots,s$ if the prime
$l_i$ is totally split in every field $K_n(l_j)$ with $i \neq j$. 
This can be arranged by choosing the sequence of primes $l_1,l_2,\dots $ inductively, as
follows.  Begin with any prime $l_1$ with $l_1 \equiv 1 \mod n$.  Suppose inductively for $1 \le t < s$
that  $\{l_1,l_2,\dots, l_t\}$ are distinct 
primes congruent to 1 modulo $n$ satisfying the condition that each $l_i$ is 
totally split in every field $K_n(l_j)$ with $j \neq i$.  It would suffice to find a prime $l_{t+1}$ with 
$l_{t+1} \equiv 1 \pmod {nl_1l_2 \dots l_t}$ (which implies both that $l_{t+1} \equiv 1 \mod n$ and that
$l_{t+1}$ splits completely in each $K_n(l_i)$ for $i \le t$) such that also $l_1, \dots, l_t$ each split completely 
in $K_n(l_{t+1})$.  That is, we want $l_{t+1} \equiv 1 \pmod{nl_1l_2\cdots l_t}$ and, 
by equation \eqref{eq: frobinKp}, for each $1 \leq i \leq t$,
$l_i \equiv y_i^n \pmod {l_{t+1}}$ for some $y_i$. 
These conditions are satisfied for any prime  $l_{t+1}$ which splits completely in the field 
$\Q(\zeta_{nl_1l_2\cdots l_t}, l_1^{1/n},l_2^{1/n},\dots,l_t^{1/n})$, and Chebotarev's density theorem ensures that there are 
infinitely many such primes.   

The proof of Proposition \ref{lift} shows that every 
tame split decomposition configuration $(H,\mathcal S, \mathcal W)$ is the quotient of a 
tame split decomposition configuration with $G=(\Z/n\Z)^s$ (take $\eta_i = 0$ in the proof), which we have
just seen can all be realized over $\QQ$.  By Proposition  \ref{prop:quotient} this proves the following theorem.

\begin{thm} \label{thm:split}
Every tame split decomposition configuration for a finite abelian group $G$ can be
realized over $\QQ$.  In particular, every tame ramification configuration for $G$
can be realized over $\QQ$.

\end{thm}

This theorem shows that every finite abelian group arises as the Galois group of
a tamely ramified extension of $\QQ$ with a minimal set of ramifying primes (which as
mentioned in the Introduction is already an easy special case of the results in \cite{KS} and \cite{KNS}),
with the added feature that the inertia groups for the ramifying 
primes can be taken to be any collection of cyclic subgroups that minimally generate $G$, and
where the ramified primes are otherwise completely split in the extension.

If now $(G, \mathcal T, \mathcal Z)$, with $\mathcal T=\{T_1,\dots,T_s\}$ and $\mathcal Z=\{Z_1,\dots,Z_s\}$ 
is a decomposition configuration 
with $G=(\Z/n\Z)^s$, then as
noted above we have $G= T_1 \times T_2 \times \dots \times T_s$ where each $T_i$ is
cyclic of order $n$; any realization over $\QQ$ must necessarily be the composite
of fields $K_n(l_i)$ for distinct primes $l_i \equiv 1 \pmod n$, whose additional
splitting information is given by the groups $Z_i$.  
If $x_i$ is a generator for $T_i$ for $i = 1, \dots , s$, then this additional 
decomposition information is determined by choosing a Frobenius element $z_i \in Z_i$ such that
$Z_i = \langle x_i, z_i \rangle$.  Since $G = T_1 \times \dots \times T_s$, we may,
by adjusting $z_i$ by a multiple of $x_i$ if necessary, write 
$z_i$ in terms of the generators $x_j$ for $j \ne i$:
\begin{equation}  \label{eq:ziingeneral}
z_i = \prod_{j=1}^s x_j^{a_{ij}}   \quad \text{with} \quad a_{ij} \in \Z/n\Z \quad \text{and} \quad a_{ii} = 0 .
\end{equation}
Conversely, given exponents $a_{ij}$ with $a_{ii} = 0$, we can define $z_i$ by equation \eqref{eq:ziingeneral}
to give a decomposition configuration with $Z_i = \langle x_i, z_i \rangle$.  
The associated $s \times s$ matrix
\begin{equation} \label{eq:decompmatrix}
M_{ \mathcal Z} =
\begin{pmatrix}
0 & a_{12} & \dots & a_{1s} \\
a_{21} & 0 & \dots & a_{2s} \\
\vdots & \vdots & \ddots & \vdots \\
a_{s1} & a_{s2} & \dots & 0 
\end{pmatrix} , \quad a_{ij} \in \Z/n\Z, \quad a_{ii} = 0,
\end{equation} 
attached to the decomposition configuration $(G=(\Z/n\Z)^s$, $\mathcal T=\{T_1,\dots,T_s\}$,
$\mathcal Z=\{Z_1,\dots,Z_s\}$)
through equation \eqref{eq:ziingeneral}
and the choice of generators $x_1, \dots x_s$ for the $T_i$ therefore
carries all of the additional decomposition information for the configuration.
The matrix $M_{ \mathcal Z}$ is, of course, generally not unique---for example, we can multiply any row or 
column by an element in $(\Z/n\Z)^*$ (corresponding to a change of generator $x_i$ or $z_i$).
A realization of this configuration involves finding the primes $l_i$ whose Frobenius is given
by \eqref{eq:ziingeneral} and \eqref{eq:decompmatrix}.

Theorem \ref{thm:split} proves that all split decomposition configurations (whose associated matrix
\eqref{eq:decompmatrix} is simply the zero matrix) can be realized over $\QQ$.  
We shall see in the next subsection that, even for abelian groups, decomposition configurations 
that are not split may not be realizable over $\QQ$.

\subsection{Finite Abelian 2-Groups}  \label{abelian2}
 
In this subsection we consider tame decomposition configurations for abelian groups of
2-power order.  We begin with the case $G=(\Z/2\Z)^s$ of a finite elementary abelian $2$-group.

Suppose $(G, \mathcal T, \mathcal Z)$ is a tame decomposition configuration 
with $G=(\Z/2\Z)^s$, $\mathcal T=\{T_1,\dots,T_s\}$ and $\mathcal Z=\{Z_1,\dots,Z_s\}$.  In this case
the generators $x_i$ for $T_i$ are unique, as are the elements $z_i$ in equation \eqref{eq:ziingeneral},
and the elements $a_{ij}$ in the associated matrix $M_{\mathcal Z}$ in equation \eqref{eq:decompmatrix} can be 
taken to be in $\{0,1\}$.  We convert the nondiagonal elements of $M_{\mathcal Z}$ to $\{\pm 1\}$ instead of
$\{0,1\}$ by defining $S_{\mathcal Z} = \big( m_{ij}\big)$ to be the 
$s \times s$ matrix such that $m_{ii}=0$ and $m_{ij}=(-1)^{a_{ij}}$ for $i\neq j.$  

\begin{defn} 
An $s\times s$ matrix $M$ is a \defi{sign matrix} if it has diagonal entries equal to $0$, and 
off-diagonal entries equal to $\pm 1.$ 
\end{defn}

\begin{defn}
A sign matrix $M$ is a \defi{QR (quadratic residue) matrix} if  $M=\big( m_{ij}\big)$ where $m_{ij} =\big(\frac{p_i}{p_j}\big)$ is
given by the Legendre symbols for some set of distinct odd primes $\{p_1,\dots,p_s\}$.
\end{defn}
   
\begin{thm} \label{thm:exponent2}
With notation as above, the tame decomposition configuration $(G, \mathcal T, \mathcal Z)$ for $G=(\Z/2\Z)^s$
is realizable over $\Q$ if and only if the corresponding sign matrix
$S_{\mathcal Z}$ is a QR matrix. For $s \le 2$, every
tame decomposition configuration is realizable over $\Q$, but for
$s \geq 3$ there exist tame decomposition configurations that cannot be realized over $\Q.$ 
\end{thm}

\begin{proof} 
Since $n = 2$, a realization $K/\Q$ of $(G, \mathcal T, \mathcal Z)$ would be given by the composite field 
$K = K_2(l_1) \cdots K_2(l_s)$ for distinct odd primes $l_1, \dots , l_s$.  Here $K_2(l_i)$ is the quadratic
subfield of the ${l_i}^\textup{th}$ roots of unity, so $K_2(l_i) = \Q(\sqrt{l_i^*})$ where $l_i^* = (-1)^{(l_i-1)/2} l_i$.  
Then for $i\neq j$, $m_{ij}= +1$ in $S_{\mathcal Z}$ if and only if $a_{ij}=0$, i.e., if and only if 
$Z_i \subseteq \Gal(K/K_2(l_j))=T_1 \times \dots T_{j-1} \times T_{j+1}\times\dots \times T_s$, hence if and
only if $l_i$ splits in $K_2(l_j)$.  Since $l_i$ splits in $K_2(l_j)$ if and only if
$\big(\frac{l_j^\ast}{l_i}\big)= \big(\frac{l_i}{l_j}\big) = +1$, it follows that $m_{ij} = \big(\frac{l_i}{l_j}\big)$
so $S_{\mathcal Z}$ is precisely the QR matrix for the primes $\{l_1,\dots,l_s\}$.

Conversely, suppose $S_{\mathcal Z}$ is a QR matrix given by the Legendre symbols for the distinct odd primes $\{l_1,\dots,l_s\}$.
Then by the equivalences in the previous paragraph, the extension
$K=\Q( \sqrt{l_1^\ast},\dots,\sqrt{l_s^\ast})$ is a realization over $\QQ$ of $(G, \mathcal T, \mathcal Z)$.

In \cite{DDK} it is shown that an $s \times s$ sign matrix $S$ is a QR matrix if and only if the diagonal entries of $S^2$ consist 
of $s-k$ 
occurrences 
of $s-1$ and $k$ 
occurrences 
of $s-2k+1$ for some integer $k$ with $1 \leq k\leq s$.  If 
\begin{equation*}
S=
\begin{pmatrix*}[r]
0 &  -1  &  -1  \\
-1 & 0  &  -1  \\
1 &  1 &  0 
\end{pmatrix*}
\end{equation*}
then  $S^2$ has diagonal entries $0,0$ and $-2$ and hence is not a QR matrix. The corresponding
tame decomposition configuration $(G, \mathcal T, \mathcal Z)$ with $G=(\Z/2\Z)^3, T_i= \langle x_i \rangle , 1\leq i \leq 3, $ 
and $Z_1=\langle x_1,x_2 x_3 \rangle, Z_2=\langle x_2,x_1 x_3 \rangle $ and $Z_3=\langle x_3 \rangle $ cannot be realized over $\Q$. 
In a similar way there are, for any $s\geq 3$, tame decomposition configurations that are not realizable over $\Q.$
For $s \le 2$, all sign matrices are QR matrices, so all tame decomposition configurations 
are realizable over $\QQ$. 
\end{proof}
 
Since a tame decomposition configuration  $(G, \mathcal T, \mathcal Z)$ is invariant under permutations 
of the indices $1 \leq i \leq s,$ we only consider the matrices $S_{\mathcal Z}$ up to conjugation by 
$s \times s$ permutation matrices.  If $S_{\mathcal Z}$ is a QR matrix and 
$K=\Q( \sqrt{l_1^*},\dots,\sqrt{l_s^*})$ is a realization over $\QQ$, this corresponds to a permutation of the primes $l_i$.  

The number of such permutation classes of $s \times s$ sign matrices is at least $2^{s^2-s}/s!$ which is greater than 
$2^{s^2(1-\delta)}$ for any $\delta >0$ as $s\to \infty$ by Sterling's formula. On the other hand, as in \cite{DDK}, 
every permutation class of 
$s \times s$ QR matrices contains a (generally non-unique) block matrix of the (``reduced'') form
\begin{equation} \label{eq:reducedform}
\begin{pmatrix}
A &  B   \\
B^t & S  
\end{pmatrix}
\end{equation}
where $A$ is a $k \times k$ skew-symmetric sign matrix, $S$ is an $(s-k) \times (s-k)$ symmetric sign matrix,  
$B$ is a $k \times (s-k)$
 matrix all of whose entries are $\pm 1$ and $B^t$ is the transpose of $B.$ But every such matrix is determined 
by its entries above the diagonal and the integer $ k, 1 \leq k \leq s.$ Therefore the number of permutation classes of $s \times s$ 
QR matrices is at most 
 $s2^{(s^2-s)/2}<2^{s^2(1+\delta)/2}$ as $s \to \infty.$ Hence
 \[
 \frac{\#\{{\rm{permutation\  classes \ of}}\  s \times s \ 
{\rm{ QR\  matrices}}\}}{\#\{{\rm{permutation\  classes\  of}}\  s \times s \ {\rm{ sign\  matrices}}\}}
  < \frac{1}{2^{s^2(1-\delta)/2}}
 \]
for any $\delta >0,$  as $s \to \infty.$ Therefore the proportion of {\it realizable\/} tame decomposition 
configurations over $\Q$ to {\it all possible\/} tame decomposition configurations becomes vanishingly small as $s \to \infty.$
 
These results show that in general it is not possible in minimally tamely ramified multiquadratic extensions of $\QQ$ to specify
the further splitting of the ramified primes arbitrarily (even rapidly less possible as the number of
ramifying primes increases)---this splitting data must give rise to a QR matrix.  The next result
shows that it {\it is\/} possible to at least specify the inertia {\it indices\/} for the ramifying primes, i.e., in the
usual terminology, we may specify ``$f = 1$ or 2'' arbitrarily for the $s$ tamely ramified primes:

\begin{cor} \label{cor:exponent2index}
For any integer $r$ with $0 \le r \le s$, there is a multiquadratic extension $K/\Q$ of degree $2^s$ in which precisely $s$
primes ramify and ramify tamely (so have ramification index $e = 2$), and precisely $r$ of the ramified primes have 
inertial degree $f$ equal to 2 and the remaining $s-r$ ramified primes are otherwise totally split in $K$ (i.e., have $f = 1$). 
\end{cor}
 
\begin{proof}
As previously mentioned, it is shown in \cite{DDK} that every QR-matrix is permutation equivalent to an $s\times s$ matrix of the form 
in equation \eqref{eq:reducedform}.  Let $S = \big( m_{ij}\big)$ be the $s\times s$ 
symmetric sign matrix whose first row has $r$ entries equal to $-1$, and $s-r$ entries which are $+1$ and all other entries 
above the diagonal are $+1$.  
Then $S$ is a QR-matrix, so by Theorem~\ref{thm:exponent2} the corresponding tame decomposition configuration
is realizable by a multiquadratic extension $K/\Q.$  
In this realization, the inertial degree for the $i^\textup{th}$ prime, $f_i$, is 2 (i.e., $\vert Z_i / T_i \vert = 2$) if and 
only if $z_i \notin T_i$, i.e., if and only if the $i^\textup{th}$ row of $S$ contains an entry equal to $-1$ 
(by equation \eqref{eq:ziingeneral} since $m_{ij}=(-1)^{a_{ij}}$).
It follows from the form of the matrix $S$ that precisely $r$ of the ramified primes have $f_i = 2$ and the remaining ramified
primes have $f_i = 1$.  
\end{proof}
 
\begin{remark}
This Corollary shows in particular that the question of realizing tame decomposition configurations is more precise than the
simpler question of specifying the ramification index $e$ and the inertial degree $f$ for each of the (tamely) ramifying primes. 
\end{remark}

Theorem \ref{thm:exponent2} shows that, for an elementary abelian 2-group of rank at least 3, not every tame decomposition configuration 
can be realized over $\QQ$.  It follows by Proposition \ref{prop:quotient} that 
if $G$ is {\it any\/} finite abelian group with 2-rank at least 3 then not every 
tame decomposition configuration can be realized over $\QQ$, since there are quotients of such configurations that cannot be
realized.

Suppose now that $G=\Z/4\Z \times \Z/2\Z$, so that $s=2.$  For a tame decomposition configuration  $(G, \mathcal T, \mathcal Z)$, 
we must have (after a suitable ordering),  $T_1=\langle x_1 \rangle $ with $x_1 \in G$ of order $4$, and $T_2=\langle x_2 \rangle $ 
with $x_2 \notin T_1.$ 
Suppose $K/\Q$ is a realization over $\QQ$ of  $(G, \mathcal T, \mathcal Z)$ and let $p_i$ be primes of $\Z$ whose inertia 
groups are the images 
of $T_i$ in $\Gal(K/\Q).$  Since $x_1$ has order 4, it
follows that $p_1 \equiv 1 \pmod 4$ (the ramification is tame, so $T_1$ is isomorphic to a subgroup of the multiplicative group
$(\ZZ/ p_1 \ZZ)^*$ of the residue field).  
If $F$ is the (unique) biquadratic subfield of 
$K$ (the fixed field of $x_1^2$), 
then $F$ is ramified only at the two primes $p_1$ and $p_2$, so $ F=\Q(\sqrt {p_1^*},\sqrt{p_2^*})$ with
$p_1^*=p_1$ and $p_2^* = (-1)^{(p_2 - 1)/2} p_2$ as usual. 
But then $p_1$ splits in $\Q(\sqrt {p_2^*})$ if and only if $p_2$ splits in $\Q(\sqrt {p_1^*})$ since 
$\big(\frac {p_2^*} {p_1}\big)=\big(\frac {p_1} {p_2}\big)=\big(\frac {p_1^*} {p_2}\big)$, so the decomposition of $p_1$ and
$p_2$ cannot be chosen independently.  
It follows that there exist some tame 
decomposition configurations $(G, \mathcal T, \mathcal Z)$ for $G=\Z/4\Z \times \Z/2\Z$
that are not realizable over $\Q$.  For an explicit example, 
the configuration $Z_1=T_1$ and $Z_2=G$ is not realizable over $\QQ$ (it requires that $p_1$ split in $\Q(\sqrt {p_2^*})$
but $p_2$ to be inert in $\Q(\sqrt {p_1^*})$).  This also shows that not all ramification and inertial indices
are possible---it is not possible to have $e = 4$, $f = 1$ (i.e., ramified of degree
4 and otherwise totally split) for one prime and $e = 2$, $f = 4$ (i.e., ramified of degree 2 and otherwise completely inert) 
for the other.  It is easy to check that this reciprocity condition is the statement that
$Z_1 = T_1$ if and only if $Z_2 \le \langle T_2, x_1^2 \rangle$.

If we write $G = \ZZ / 4 \ZZ \times \ZZ / 2 \ZZ = \langle x_1 \rangle \times \langle y \rangle$ then
up to an isomorphism of $G$ or an interchange of $\{ T_1, Z_1 \}$ and $\{ T_2, Z_2 \}$, there are nine
possible tame decomposition configurations for $G$.   
Searching explicit examples in the number field data base \cite{JR} shows that all configurations 
satisfying the reciprocity condition 
can be realized by a tame decomposition extension over $\Q$, so in fact this is the only obstruction 
to finding a realization in this case. 
In Table \ref{table:Z4Z2} we list the nine possible tame decomposition configurations 
$\{ T_1, Z_1 \}$ and $\{ T_2, Z_2 \}$  for  $\ZZ / 4 \ZZ \times \ZZ / 2 \ZZ$.
For each realizable configuration we give a polynomial and primes $p_1$ and $p_2$
(with ramification index $e$ and inertial degree $f$) realizing the configuration.

\begin {table}[ht]    %  
\begin{center}
  \begin{tabular}{c  c  c  c | c  c | p{5.3cm} }
$T_1$ & $Z_1$ & $T_2$ & $Z_2$  & $p_1 \ (e,f)$ & $p_2 \ (e,f)$ &  \text{realization (when possible)} \\ \hline
\multicolumn{1}{|r}{$\gp{x_1}$} & $\gp{x_1}$    &   $\gp{y}$  &   $\gp{y}$  & 13 (4,1) & 3 (2,1)
             & \multicolumn{1}{p{5.3cm}|}{$x^8 - x^7 - x^6 - 10x^5 + 5x^4 + 14x^3 + 10x^2 + 12x + 9$} \\ \hline
\multicolumn{1}{|r}{$\gp{x_1}$} & $\gp{x_1}$    &   $\gp{y}$  &   $\gp{y,x_1^2}$   & 37 (4,1)   & 3 (2,2)
             & \multicolumn{1}{p{5.3cm}|}{$x^8 - x^7 - 4 x^6 + 9 x^5 - 31 x^4 + 63 x^3 - 196 x^2 - 343 x + 2401$} 
             \\ \hline
\multicolumn{1}{|r}{$\gp{x_1}$} & $\gp{x_1}$    &   $\gp{y}$  &   $G$   &   &  
             & \multicolumn{1}{p{5.3cm}|}{not realizable} 
             \\ \hline
\multicolumn{1}{|r}{$\gp{x_1}$} & $G$    &   $\gp{y}$  &   $\gp{y}$   &    &   
             & \multicolumn{1}{p{5.3cm}|}{not realizable} 
             \\ \hline
\multicolumn{1}{|r}{$\gp{x_1}$} & $G$    &   $\gp{y}$  &   $\gp{y,x_1^2}$   &    &   
             & \multicolumn{1}{p{5.3cm}|}{not realizable} 
             \\ \hline
\multicolumn{1}{|r}{$\gp{x_1}$} & $G $    &   $\gp{y}$  &   $G$   & 5 (4,2) & 3 (2,4)
             & \multicolumn{1}{p{5.3cm}|}{$x^8 - x^7 + x^5 - x^4 + x^3 - x + 1$}
             \\ \hline
\multicolumn{1}{|r}{$\gp{x_1}$} & $\gp{x_1}$    &   $\gp{x_1y}$  &   $\gp{x_1y}$  & 5 (4,1) & 29 (4,1)
             & \multicolumn{1}{p{5.3cm}|}{$x^8 - x^7 - 47 x^6 + 40 x^5 + 581 x^4 - 655 x^3 - 1603 x^2 + 1968 x + 36$}
             \\ \hline
\multicolumn{1}{|r}{$\gp{x_1}$} & $\gp{x_1}$    &   $\gp{x_1y}$  &   $G$   &   &    
             & \multicolumn{1}{p{5.3cm}|}{not realizable} 
             \\ \hline
\multicolumn{1}{|r}{$\gp{x_1}$} & $G$    &   $\gp{x_1y}$  &   $G$    & 5 (4,2) & 13 (4,2) 
             & \multicolumn{1}{p{5.3cm}|}{$x^8 - x^7 - 21 x^6 + 18 x^5 + 89 x^4 - 19 x^3 - 89 x^2 - 38 x - 4$}
             \\  \hline            
  \end{tabular}
\\[5pt] 
 \caption {Decomposition Configurations for $\Z/4\Z \times \Z/2\Z$} \label{table:Z4Z2} 
%  \\[5pt] % You can adjust how far below the table the text should appear
%  Is just like a caption
 \end{center}
\end {table}

As before, since there are configurations that cannot be realized over $\QQ$,
it follows that {\it any\/} finite abelian group $G$ whose 2-primary part has rank 2 but not exponent 2 will have
tame decomposition configurations that cannot be realized over $\QQ$, since it has quotients that cannot be realized.  

Finally, if 
$G=\Z/2^n\Z$, then $s=1$ and the only tame decomposition configuration  $(G, \mathcal T, \mathcal Z)$ is $T=Z=G$ 
which is realizable over $\Q$ by the subfield $K_{2^n}(l)$ of degree $2^n$ of the cyclotomic field of $l^\textup{th}$ roots of
unity for any prime $l \equiv 1 \pmod {2^n}$.

\subsection{A Reciprocity Theorem} \label{sec:reciprocitylaw}

Before considering finite abelian groups further, we first prove
a reciprocity theorem which may be of independent interest.

Fix an integer $n \ge 1$, a prime $p \equiv 1$ mod $n$, and distinct primes $l_1, \dots , l_s$
different from $p$ and prime to $n$.

\begin{lem} \label{lem:equalperps}
Suppose $\vec a = (a_1, \dots , a_s)$ and $\vec b = (b_1, \dots , b_s) \in \ZZ^s$  have the property that
for any $(A_1, \dots, A_s) \in \ZZ^s$, 
$$
(a_1, \dots , a_s) \cdot (A_1, \dots , A_s) \in n\ZZ \mathrm {\ if \ and \ only \ if \ } (b_1, \dots , b_s) \cdot (A_1, \dots , A_s) \in n\ZZ .
$$
Then there is an integer $u$ relatively prime to $n$ so that $ u (a_1, \dots , a_s) \equiv (b_1, \dots , b_s) \mod n$, i.e.,
$u a_i \equiv b_i \mod n$ for all $i =1,\dots,s$.

\end{lem}

\begin{proof}
For any $\vec a$, the map $f_{\vec a} : \ZZ^s \rightarrow Z/nZ$ defined by the usual 
dot product mod $n$, $f_{\vec a} (\vec z) = \vec a \cdot \vec z \mod n$, is a homomorphism of abelian groups.
The assumption of the Lemma is that $ \ker f_{\vec a}= \ker f_{\vec b}$, so that
$ \ZZ^s / \ker f_{\vec a}    =    \ZZ^s / \ker f_{\vec b}$.  Since
$ \ZZ^s / \ker f_{\vec a}  \simeq d \ZZ / n \ZZ$ for some divisor $d$ of $n$,
also $ \ZZ^s / \ker f_{\vec b} \simeq d \ZZ / n \ZZ$.  Any automorphism of 
$d \ZZ / n \ZZ$ is given by multiplication
by some element $u$ relatively prime to $n$, so $ u f_{\vec a}   (\vec z)    =   f_{\vec b}   (\vec z) $
for all $\vec z$. Evaluating this for $\vec z$ the usual vectors $\epsilon_i = (0,...,0,1,0,...,0)$ gives
$u (a_1, \dots , a_s) \equiv (b_1, \dots , b_s) \mod n$.
\end{proof}

Let $K_n(p)$ be the degree $n$ subfield of the cyclotomic field of 
$p^\textup{th}$ roots of unity as in Definition \ref{def:Kp} and the following discussion.

Let $F$ be the cyclotomic field of $n^\textup{th}$ roots of unity and let 
$$
L = F (l_1^{1/n}, \dots , l_s^{1/n}) .
$$

The field $L$ is Galois over $\QQ$ with Galois group isomorphic to the
semidirect product of $\Gal (L/F)$ by $\Gal(F/\QQ)$. 
The abelian normal subgroup 
$\Gal (L/F)$ is canonically isomorphic to $\mu_n^s$ where $\mu_n$ is the group of $n^\textup{th}$ roots of
unity under the map 
$\sigma \mapsto (\dots, \eta_i , \dots)$, 
where
$$
\sigma (l_i^{1/n}) = \eta_i (l_i^{1/n}), \quad 1 \le i \le s.
$$
Choosing a fixed primitive $n^\textup{th}$ root of unity $\zeta_n$ defines a (noncanonical) isomorphism of $\mu_n$ with
$\ZZ / n \ZZ$, inducing a (noncanonical) isomorphism of $\mu_n^s$ with $(\ZZ / n \ZZ)^s$.  The corresponding (noncanonical)
isomorphism of $\Gal (L/F)$ with $(\ZZ / n \ZZ)^s$ maps the automorphism
$ \lambda_1^{x_1} \dots \lambda_s^{x_s} $ to $(x_1, \dots , x_s)$ where 
for $i = 1, \dots , s$, the automorphism $\lambda_i \in \Gal (L/F)$ is defined by
\begin{equation} \label{eq:lamdadefs}
\lambda_i = 
\begin{cases}
    {l_i^{1/n}} \mapsto \zeta_n \ {l_i^{1/n}} &  \\
       {l_j^{1/n}} \mapsto {l_j^{1/n}} & \text{for } j \ne i . 
  \end{cases}
\end{equation}

The automorphisms $\sigma_a \in \Gal(F/\QQ)$ for $(a,n) = 1$ lift to elements of $\Gal(L/\QQ)$ by defining
$\sigma_a (l_i^{1/n}) = l_i^{1/n}$ for $i = 1, \dots, s$.  Conjugation by $\sigma_a$ on the abelian normal 
subgroup $\Gal(L/F)$ in the semidirect product is by raising to the $a^\textup{th}$ power:
$ \sigma_a (\lambda_1^{x_1} \dots \lambda_s^{x_s} )  \sigma_a^{-1} = \lambda_1^{a x_1} \dots \lambda_s^{a x_s} $.  
In particular, every subgroup of $\Gal(L/F)$ is normal in $\Gal(L/\QQ)$.

The prime $p$ splits completely in $F$ since $p \equiv 1$ mod $n$, and if $\wp$ is any one of the $\ph(n)$ distinct primes
of $F$ dividing $(p)$ then the Frobenius automorphism $\Fr_{L/F} (\wp )$ in the abelian extension $L/F$ depends only on $\wp$.  
If $\sigma_a ( \wp) $ ($\sigma_a \in \Gal(F/\QQ)$) is any other prime of $F$ lying over $(p)$, then 
$\Fr_{L/F} (\sigma_a ( \wp )) = \Fr_{L/F} (\wp )^a$.

\begin{thm} \label{thm:reciprocitytheorem}
(Reciprocity Theorem) Let $\Fr_{L/F} (\wp) \in \Gal(L/F)$ be the Frobenius element for any prime of $F$ lying above $(p)$ and suppose
$ \Fr_{L/F} (\wp) = \lambda_1^{a_1} \dots \lambda_s^{a_s}$ where the $\lambda_i \in \Gal(L/F)$ are as in equation \eqref{eq:lamdadefs}.
Then there is a generator $\tau$ for $\Gal (K_n(p)/\QQ)$ with $ \Fr_{K_n(p)/\QQ} (l_i) = \tau^{a_i}$ for all $i = 1, \dots, s.$
\end{thm}

\begin{proof}
If $\Fr_{L/F} (\wp) = \lambda_1^{a_1} \dots \lambda_s^{a_s} $, then
for any $A_1, \dots , A_s \in \ZZ$, 
\begin{equation} % \label{eq:galoisaction}
\Fr_{L/F} (\wp) ( (l_1^{1/n})^{A_1} \dots (l_s^{1/n})^{A_s} ) 
= \zeta_n^{a_1 A_1 + \dots + a_s A_s} (l_1^{1/n})^{A_1} \dots (l_s^{1/n})^{A_s}  .
\end{equation}
It follows that $(a_1, \dots, a_s)\cdot (A_1, \dots, A_s) \equiv  0 \mod n$ if and only if
the decomposition field for $\wp$ in the abelian extension $L/F$ contains 
the element $(l_1^{1/n})^{A_1} \dots (l_s^{1/n})^{A_s} $.

Now, the element 
$(l_1^{1/n})^{A_1} \dots (l_s^{1/n})^{A_s} $ lies in the decomposition field for $\wp$ 
if and only if $\wp$ splits completely in the field $F((l_1^{1/n})^{A_1} \dots (l_s^{1/n})^{A_s})$,
and since $\wp$ is a degree one prime of $F$, this is true 
if and only if the polynomial 
$x^n - {l_1}^{A_1} \dots {l_s}^{A_s}$ has a root mod $p$, i.e., if and only if 
${l_1}^{A_1} \dots {l_s}^{A_s}$ is an $n^{\text{th}}$ power mod $p$.
This is equivalent to the statement that 
$\sigma_{l_1}^{A_1} \dots \sigma_{l_s}^{A_s}$ projects to the identity in
$\Gal(K_n(p)/\QQ)$; since $\sigma_{l_i}$ projects to the
Frobenius automorphism $\Fr_{K_n(p)/\QQ} (l_i)$ in $\Gal(K_n(p)/\QQ)$, this is in turn
equivalent to the statement that 
$\Fr_{K_n(p)/\QQ} (l_1)^{A_1} \dots \Fr_{K_n(p)/\QQ} (l_s)^{A_s} = 1$ in $\Gal(K_n(p)/\QQ)$.
If $l_i = g^{b_i} \mod p$ for $g$ a primitive root modulo $p$, 
then by equation \eqref{eq: frobinKp}, $\Fr_{K_n(p)/\QQ} (l_i) = \tau_g^{b_i}$, so 
$\Fr_{K_n(p)/\QQ} (l_1)^{A_1} \dots \Fr_{K_n(p)/\QQ} (l_s)^{A_s} = 1$
if and only if
$\tau_g^{b_1 A_1 + \dots + b_s A_s} = 1$ in $\Gal(K_n(p)/\QQ)$, i.e., if and only if
$(b_1, \dots , b_s) \cdot (A_1, \dots , A_s) \in n \ZZ$.

Hence, $(a_1, \dots , a_s) \cdot (A_1, \dots , A_s) \in n \ZZ$ if and only if
$(b_1, \dots , b_s) \cdot (A_1, \dots , A_s) \in n \ZZ$.  By Lemma \ref{lem:equalperps}, there is
an integer $u$ relatively prime to $n$ with 
$u (a_1, \dots , a_s) \equiv (b_1, \dots , b_s) \mod n$.  
Since $\Fr_{K_n(p)/\QQ} (l_i) = \tau_g^{b_i} = \tau_g^{u a_i}$, the theorem follows with
$\tau = \tau_g^u$.  
\end{proof}

\begin{remark}
There are $\ph(n)$ distinct primes $\wp$ above $(p)$ in $F$, and there are $\ph(n)$ distinct generators
for $\Gal (K_n(p)/\QQ)$.  Distinct primes $\wp$ correspond to distinct generators $\tau$ in 
Theorem \ref{thm:reciprocitytheorem}: if $\wp$ corresponds to $\tau$, then
$\sigma_a (\wp)$ for $\sigma_a \in \Gal(F/\QQ)$ corresponds to $\tau^a$.  Note in particular that the subgroup
$\langle \Fr_{L/F} (\wp) \rangle \le \Gal(L/F)$ depends only on $p$ and not on the choice of $\wp$ dividing $(p)$ in $F$.  
\end{remark}

\begin{remark}
When $n = 2$, Theorem \ref{thm:reciprocitytheorem} is just quadratic reciprocity for $p$ and the primes $l_i$, $i = 1, \dots, s$.
For larger values of $n$, the theorem carries more information than simply the appropriate $n^\textup{th}$ power
reciprocity for the individual primes.  For example, if $n= 3$ and $s = 2$ the theorem not only considers whether
$l_1$ and $l_2$ are cubes modulo $p$, but also when they are not cubes whether they lie in the {\it same\/} (or the
{\it inverse\/}) cubic residue class modulo $p$. 
\end{remark}

\subsection{Finite Abelian Groups of Odd Order}   \label{abelianodd}

In this subsection we show that, unlike the case of abelian 2-groups where quadratic reciprocity intervened,
for finite abelian groups of odd order there are no constraints to realizing tame decomposition configurations over $\QQ$:

\begin{thm} \label{thm:oddabelian}
Every tame decomposition configuration $(G,\mathcal T, \mathcal Z)$ with $G$ an abelian group of odd order is realizable over $\Q.$
\end{thm}

The method of proof will be similar to the proof of Theorem \ref{thm:split}: by Propositions \ref{prop:quotient} and \ref{lift} 
(noting in the proof of the latter than 
$n$ can be taken odd if $H$ has odd exponent) it suffices to prove that all 
tame decomposition configurations are realizable over $\Q$ for the group
$G=(\Z/n\Z)^s$ when $n$ is odd, which we do using subfields of appropriate cyclotomic fields, whose
existence will be proved using the Reciprocity Theorem of the previous subsection.

For the remainder of this subsection let $n$ be an odd positive integer.  

A tame decomposition configuration for $G=(\Z/n\Z)^s$ is encoded in a matrix
$M_{\mathcal Z}$ as in equation \eqref{eq:decompmatrix}.  To prove Theorem \eqref{thm:oddabelian} we must show
that any such matrix containing arbitrary elements
of $\Z/n\Z$ in the off-diagonal positions arises from the decomposition information for the distinct primes
$l_1, \dots , l_s$ in the composite extension $K = K_n(l_1) \dots K_n(l_s)$ (where $K_n(l_i)$ is the field in
Definition \ref{def:Kp}), which we now make explicit. 

As in equation \eqref{eq:Kpisom}, let $g_i$ be a primitive root modulo $l_i$ and 
let $\tau_{g_i}$ be the corresponding generator for $\Gal(K_n(l_i)/\QQ)$, 
viewed as an element in $\Gal(K/\QQ)$ with $\tau_{g_i}$ acting
trivially on each $K_n(l_j)$ with $j \ne i$ (which amounts to choosing $g_i \equiv 1 \pmod {l_j}$ for $j \ne i$), so that
$\Gal(K/\QQ) = \langle \tau_{g_1} \rangle \times \dots \times \langle \tau_{g_s} \rangle$.  
The decomposition group for the prime $l_i$ in $\Gal(K/\QQ)$ is
generated by $\tau_{g_i}$ and the Frobenius automorphism for $l_i$ in 
$\Gal(K/K_n(l_i) ) = \langle \tau_{g_1} \rangle \times \dots \langle \tau_{g_{i-1}} \rangle \times \langle \tau_{g_{i+1}} \rangle 
\dots \times \langle \tau_{g_s} \rangle$.
This Frobenius automorphism is $\tau_{g_1}^{a_{i1}} \dots \tau_{g_s}^{a_{is}}$ where $\tau_{g_j}^{a_{ij}}$ is the restriction to 
$K_n(l_j)$ of the automorphism $\sigma_{l_i} \in \Gal( \QQ(\zeta_{l_j})/\QQ)$ for $j \ne i$ and $a_{ii} = 0$.  
By equation \eqref{eq: frobinKp} the $a_{ij}$ are given by $ l_i \equiv g_j^{a_{ij}} \pmod {l_j} $.  

It follows that the matrix encoding the decomposition information for $K/\QQ$ is the matrix
$\big( a_{ij}\big)$ whose diagonal entries $a_{ii}$ are zero and whose off-diagonal entries $a_{ij}$ for $i \ne j$ are determined by
$ l_i \equiv g_j^{a_{ij}} \pmod {l_j} $.  To prove Theorem \ref{thm:oddabelian} we must show that we may arrange for arbitrary elements
of $\Z/n\Z$ in the off-diagonal positions by choosing the primes $l_1, \dots , l_s$ appropriately.

We proceed by induction on $s$, the case $s = 1$ being trivial.  
Suppose the result is true for $s$ and let $S$ be an $(s+1)\times (s+1)$ matrix which we write as
\begin{equation} 
S= 
\begin{pmatrix}
0 & b _1 & b_2 & \dots & b_s   \\
a_1 & 0 &  a_{12}  &   \dots & a_{1s}  \\
a_2 & a_{21} & 0  &   \dots & a_{2s} \\
\vdots & \vdots &  \vdots & \ddots & \vdots \\
a_s & a_{s1} & a_{s2} &  \dots & 0 
\end{pmatrix}
\end{equation} 
and let $S'$ be the $s\times s$ minor given by 
\begin{equation} 
S'= 
\begin{pmatrix}
0 &  a_{12}  &   \dots & a_{1s}  \\
a_{21} & 0  &   \dots & a_{2s} \\
\vdots & \vdots &   \ddots & \vdots \\
a_{s1} & a_{s2} &  \dots & 0\\
\end{pmatrix} .
\end{equation} 
By induction, there are distinct primes $l_i \equiv 1 \pmod n$ and primitive roots $g_i$ mod $l_i$ for $i = 1, \dots, s$
so that the decomposition configuration for the composite field
$K = K_n(l_1) \dots K_n(l_s)$ gives the matrix $S'$, and we must find a prime $p \equiv 1 \pmod n$ 
distinct from $l_1, \dots, l_s$ so that
the decomposition configuration for the composite extension $K_n(p) K$ is the matrix $S$.  
The conditions on $p \equiv 1 \pmod n$ in order that $\{ p, l_1, \dots, l_s \}$ produces the matrix $S$ are then
\begin{align}
& \text{1. $\Fr_{K_n(l_i)/\QQ} (p) = \tau_{g_i}^{b_i}$ for $i = 1, \dots, s$, and }   \label{eq:frobcondition1}  \\
& \text{2. $\Fr_{K_n(p)/\QQ} (l_i) = \tau_{g}^{a_i}$ for $i = 1, \dots, s$ 
for some primitive root $g$ mod $p$.} \label{eq:frobcondition2}
\end{align}

As in subsection \ref{sec:reciprocitylaw}, let $F = \QQ(\zeta_n)$ and $L = F( l_1^{1/n}, \dots , l_s^{1/n})$.  
Since $n$ is odd, the Galois extensions $L$ and $K$ are 
linearly disjoint with Galois composite $L K$.
Extending the elements in $\Gal (L / \QQ)$ to $\Gal (L K/ \QQ)$ by having them act
trivially on $K$, and similarly for $\Gal (K / \QQ)$, we can identify $\Gal (L K/ \QQ)$ with
$\Gal (L / \QQ) \times \Gal ( K / \QQ)$.

By the remarks following equation \eqref{eq:lamdadefs}, the conjugacy class in $\Gal (L  K/ \QQ)$ of the element 
$( \lambda_1^{a_1} \dots \lambda_s^{a_s} , \tau_{g_1}^{b_1} \dots \tau_{g_s}^{b_s} )$ 
consists of the elements  $( \lambda_1^{a a_1} \dots \lambda_s^{a a_s} , \tau_{g_1}^{b_1} \dots \tau_{g_s}^{b_s} )$
where the $ \lambda_1^{a a_1} \dots \lambda_s^{a a_s} $ are the generators for the subgroup 
$\langle \lambda_1^{a_1} \dots \lambda_s^{a_s}  \rangle $ in $\Gal (L / \QQ)$.

By Chebotarev's density theorem, there exists a prime $p$ whose Frobenius automorphisms 
in $\Gal (L K / \QQ)$ 
give the conjugacy class of 
% are the elements in the conjugacy class of
$( \lambda_1^{a_1} \dots \lambda_s^{a_s} , \tau_{g_1}^{b_1} \dots \tau_{g_s}^{b_s} )$.
Then $\Fr_{K_n(l_i)/\QQ} (p) = \tau_{g_i}^{b_i}$ for $i = 1, \dots, s$, so the conditions in \eqref{eq:frobcondition1} are satisfied for $p$.  
Also, there is a prime $\mathfrak p$ lying over 
$p$ in $L$ with $\Fr_{L/\QQ} (\mathfrak p) = \lambda_1^{a_1} \dots \lambda_s^{a_s}$.  Since $\lambda_1^{a_1} \dots \lambda_s^{a_s}$
is trivial on the subfield $\QQ(\zeta_n)$, $p$ splits completely in $\QQ(\zeta_n)$, i.e., $p \equiv 1 \pmod n$, and if
$\wp = \mathfrak p \cap \QQ(\zeta_n)$ then $\Fr_{L/F} (\wp) = \lambda_1^{a_1} \dots \lambda_s^{a_s}$ in the abelian extension $L/F$.

By the Reciprocity Theorem \ref{thm:reciprocitytheorem}, 
$\Fr_{L/F} (\wp) = \lambda_1^{a_1} \dots \lambda_s^{a_s}$ implies $ \Fr_{K_n(p)/\QQ} (l_i) = \tau^{a_i}$, $1 \le i \le s$,
for some generator $\tau$ of $\Gal (K_n(p)/\QQ)$.  If $g$ is a primitive root mod $p$ so that
$\tau = \tau_g$ as in \eqref{eq:Kpisom}, this shows the conditions in \eqref{eq:frobcondition2} are also satisfied for $p$, completing the
proof of Theorem \ref{thm:oddabelian} by induction.   \hfill{$\square$}

\begin{remark}
We note that $n$ odd was needed to ensure $L = \QQ(\zeta_n, l_1^{1/n}, \dots, l_s^{1/n})$ and $K = K_n(l_1) \dots K_n(l_s)$ are
linearly disjoint, which allowed us to find a prime $p$ that satisfied both \eqref{eq:frobcondition1} and
\eqref{eq:frobcondition2} simultaneously.  Since the fields $\QQ( \sqrt{l_1}, \sqrt{l_2})$ and
$\QQ( \sqrt{l_1^*}, \sqrt{l_2^*})$ always have a nontrivial intersection, the fields $L$ and $K$ are never
linearly disjoint when $n$ is even. This reflects the intervention of quadratic reciprocity constraints preventing 
some decomposition configurations from being realizable over $\QQ$ for abelian groups of even order as in the previous subsection.
\end{remark}

\subsection{Finite Abelian Groups with Cyclic or $\ZZ/2\ZZ \times \ZZ/2\ZZ$ Sylow-2 subgroup}   \label{abelianSylow2}
We have seen that if the abelian group $G$ has trivial Sylow-2 subgroup then all tame decomposition configurations
can be realized over $\QQ$, and that if the Sylow-2 subgroup has rank at least 3 or has rank 2 but not exponent 2 then there
are configurations that cannot be realized.  In this subsection we complete the analysis for finite abelian groups by
showing that in the two remaining cases, namely where the Sylow-2 subgroup of the abelian group is either
cylic or isomorphic to $\ZZ/2\ZZ \times \ZZ/2\ZZ$, then again all tame decomposition configurations
can be realized over $\QQ$:

\begin{thm} \label{thm:Sylow2abelian}
Every tame decomposition configuration $(G,\mathcal T, \mathcal Z)$ with $G$ an abelian group with
cyclic or $\ZZ/2\ZZ \times \ZZ/2\ZZ$ Sylow-2 subgroup is realizable over $\QQ.$
\end{thm}

In both cases the argument will be by induction on the number of generators of $G$, with the inductive hypothesis
used to provide an `initial realization' for a tame decomposition configuration related to 
the desired configuration.  A Chebotarev density theorem argument will show the existence of an additional prime and
corresponding abelian extensions which, when taken together with the initial realization, can be used to construct a realization of
the desired decomposition configuration. 
While somewhat more technically detailed, the arguments are fundamentally similiar to those in the previous subsection,
so we indicate the proofs more briefly.  

We first note a variant of Proposition \ref{lift}.  Suppose $(H,\mathcal S, \mathcal W)$ and $(G,\mathcal T, \mathcal Z)$ 
are as in Proposition \ref{lift}.  Write $H = H_2 \times H_{\text {odd }}$ where $H_2$ is the Sylow-2 subgroup of $H$ and
$H_{\text {odd }}$ has odd order and write $G = G_2 \times G_{\text {odd }}$ similarly.  The surjection $\pi$ in the proof of
Proposition \ref{lift} restricted to $G_2$ gives a surjection to $H_2$, and if $J_2$ denotes the kernel of this latter
map, then $\pi$ induces a surjective homomorphism $G/J_2 \rightarrow H$, and the induced minimal tame decomposition configuration
on $G/J_2$ has the minimal tame decomposition configuration $(H,\mathcal S, \mathcal W)$ on $H$ as quotient. 

It follows that to show all tame decomposition configurations can be realized over $\QQ$ for finite abelian
groups $G$ whose Sylow-2 subgroup is cyclic or isomorphic to $\ZZ/2 \ZZ \times \ZZ/2 \ZZ$, it suffices to prove all
configurations can be realized in the following two cases:
\begin{enumerate}
\item[{(1)}]
$G \isom \ZZ/2^A\ZZ \times (\ZZ/n\ZZ)^s $, where $A \ge 1$, $s \ge 0$, and $n \ge 1$ is odd, and 

\item[{(2)}]
$G \isom (\ZZ/2 \ZZ \times \ZZ/2 \ZZ) \times (\ZZ/n\ZZ)^s $, where $s \ge 0$ and $n \ge 1$ is odd.
\end{enumerate}
By previous results, we may assume $s \ge 1$.  We handle each case in turn.

\subsubsection{Finite Abelian Groups with Cyclic Sylow-2 subgroup}
Suppose 
\begin{equation} \label{eq:cylics}
G =\gp{\sigma} \times \gp{x_1} \times \dots \times \gp{x_s} \isom \ZZ/2^A\ZZ \times (\ZZ/n\ZZ)^s
\end{equation}
with $A \ge 1$, $s \ge 1$ and $n \ge 1$ is odd.  Then a
tame decomposition configuration is given by a choice, for $i = 1,2, \dots, s$, of inertia groups
\begin{equation} \label{eq:Tconfigcyclic2}
T_i = \gp{\sigma^{2^{a_i}}} \times \gp{x_i} , \quad \text{ where $a_1 = 0$}, 
\end{equation}
that generate $G$, and a choice
\begin{equation} \label{eq:Zconfigcyclic2}
Z_i = \gp{\sigma^{2^{b_i}}} \times \gp{x_i, z_i} \quad \text{with $b_i \le a_i$ and} 
\quad z_i = \prod_{j=1}^s x_j^{a_{ij}} \quad (a_{ii} = 0) , 
\end{equation}
for the decomposition groups.  

If $s = 1$ then $G$ is cyclic, a case already considered.  Assume $s \ge 2$ and by induction 
that all tame decomposition configurations can be realized for groups as in \eqref{eq:cylics} of rank $s-1$.  
Let $a = \min (a_2, a_3, \dots , a_s)$ and consider the group of rank $s-1$ defined by
$$
\gp{T_2, T_3, \dots , T_s } = 
\gp { \sigma^{2^{a}} ,x_2,x_3, \dots , x_s}  \isom \ZZ / 2^{A - a} \ZZ \times (\ZZ / n \ZZ)^{s-1}.
$$
For $i = 2,3, \dots, s$ set $\tilde {z_i} = z_i / x_1^{a_{i1}} $
and consider the configuration defined by
\begin{align*}
\widetilde{T_i} & = T_i  = \gp{ \sigma^{2^{a_i}} , x_i} \isom \ZZ / 2^{A - a_i} \ZZ \times \ZZ / n \ZZ . \quad i = 2,3, \dots, s \\
\widetilde{Z_i} & = \gp{ \sigma^{2^{a_i}} , x_i, \tilde z_i} .
\end{align*}
By induction, this configuration has a realization over $\QQ$ which is the composite of the extension
$K_n(l_2) \cdots K_n (l_s)$ for primes $l_2, \dots , l_s$ and a cyclic extension $K_0$ of degree 
$2^{A - a}$.  For each $i = 2,3, \dots, s$, the
prime $l_i$ is ramified of degree $2^{A - a_i}$ in $K_0$ and since this is the precise
power of 2 in $\order{\widetilde{Z_i}}$, the prime $l_i$ is otherwise totally split in $K_0$.  

We now find another prime $p$ to adjust this `initial configuration' to realize the decomposition \eqref{eq:Tconfigcyclic2} and
\eqref{eq:Zconfigcyclic2}.  

Choose the prime $p$ distinct from $l_2, \dots , l_s$ so that $p \equiv 1 $ modulo $2^A$ and so that
for $i = 2, 3, \dots , s$ the residue degree of $l_i$ in $K_{2^A}(p)$ is $2^{A - b_i}$ and $l_i$ is otherwise split, i.e.,
so that $Z_{K_{2^A}(p)/\QQ}(l_i) = \gp {{g_1}^{\! 2^{b_i}}}$ if $\Gal(K_{2^A}(p) / \QQ) = \gp{g_1}$.
By the Reciprocity Theorem \ref{thm:reciprocitytheorem}, this is 
a choice of Frobenius for $p$ in the extension $\QQ(\zeta_{2^A})( l_2^{1/2^A}, \dots ,  l_s^{1/2^A} )$.

Write $\Gal(K_0 / \QQ) = \gp{g_2}$, so the composite of $K_0$ and $K_{2^A}(p)$ has Galois 
group $\Gal(K_0 K_{2^A}(p)/\QQ) = \Gal(K_{2^A}(p)/\QQ) \times \Gal(K_0/\QQ) = \gp{g_1} \times \gp{g_2}$ where we 
lift $g_1$ by having it act trivially on $K_0$ and we lift $g_2$ similarly.  Define the field
$F$ to be the subfield of $K_0 K_{2^A}(p)$ fixed by the subgroup generated by ${g_1}^{\! 2^{a}} g_2$.

The field $F$ is a cyclic extension of $\QQ$ of degree $2^A$.  If we let $\Gal(F/\QQ) = \gp{\sigma}$,
a straightforward computation of the images in $\Gal (F/\QQ)$ of the inertia and 
decomposition groups for $p, l_2, \dots , l_s$ in $K_0 K_{2^A}(p)$ shows that
$$
T_{F/\QQ} (p) = \gp{\sigma}, \quad Z_{F/\QQ} (p) = \gp{\sigma}, \quad T_{F/\QQ} (l_i) = \gp{\sigma^{2^{a_i}}}, 
\quad Z_{F/\QQ} (l_i) = \gp{\sigma^{2^{b_i}}} ,
$$
i.e., $F$ realizes the even part of the tame decomposition 
configuration \eqref{eq:Tconfigcyclic2} and \eqref{eq:Zconfigcyclic2}. We may impose additional conditions
on $p$ as in the previous subsection so that the field $ K_n(p) K_n(l_2) \dots K_n(l_s)$
realizes the odd part of the configuration. 

In summary, the conditions required on the prime $p$ are the following:
\begin{enumerate}
\item[{(1)}]
$p$ is distinct from $l_2, \dots , l_s$, 

\item[{(2)}]
a choice of Frobenius for $p$ in the extension $\QQ(\zeta_{2^A})( l_2^{1/2^A}, \dots ,  l_s^{1/2^A} )$,

\item[{(3)}]
a choice of Frobenius for $p$ in the extension $K_n(l_2) \cdots K_n (l_s)$, and

\item[{(4)}]
a choice of Frobenius for $p$ in the extension $\QQ(\zeta_{n})( l_2^{1/n}, \dots ,  l_s^{1/n} )$
(which includes the condition $p \equiv 1 $ modulo $n$).

\end{enumerate}
Since the fields involved are linearly disjoint, Chebotarev's density theorem ensures the
existence of a prime $p$ satisfying all of the necessary conditions simultaneously, and for this prime, the
composite field $ F K_n(p) K_n(l_2) \dots K_n(l_s)$ gives a realization of the 
tame decomposition configuration \eqref{eq:Tconfigcyclic2} and \eqref{eq:Zconfigcyclic2}.

\subsubsection{Finite Abelian Groups with Sylow-2 subgroup $\ZZ/2\ZZ \times \ZZ/2\ZZ$}
Suppose that $G \isom (\ZZ/2 \ZZ \times \ZZ/2 \ZZ) \times (\ZZ/n\ZZ)^s $, where $s \ge 0$ and $n \ge 1$ is odd.
The cases $s = 0$ (which has already been considered in section \ref{abelian2} in any case) and 
$s = 1$ are quotients of the case $s = 2$ in the sense of Definition \ref{def:quotient},
so it suffices to prove Theorem \ref{thm:Sylow2abelian} when $s \ge 2$.  We proceed by induction on $s$.

If $s = 2$, $G = T_1 \times T_2 \isom \ZZ/(2n) \ZZ \times  \ZZ/(2n) \ZZ$, so a realization 
would be the composite $K_{2n}(p) K_{2n}(q)$ for some primes $p$ and $q$ both
congruent to 1 modulo $2n$.  To prove the existence of appropriate primes $p$ and $q$ we
can proceed as we did for abelian groups of odd order, as follows.  First 
let $q$ be any prime congruent to 3 modulo 4 and also congruent to 1 modulo $2n$.  A tame
decomposition configuration specifies the further splitting of the prime $q$ in $K_{2n}(p)$ for
a prime $p \equiv 1$ modulo $n$ and the further splitting of $p$ 
in $K_{2n}(q)$.  The splitting for $p$ is the choice of a Frobenius element for $p$ in $K_{2n}(q)$, and by 
the Reciprocity Theorem \ref{thm:reciprocitytheorem}, the
splitting for $q$ (a choice of Frobenius element for $q$ in $K_{2n}(p)$) is a choice of a
Frobenius element for $p$ in the extension $\QQ(\zeta_{2n}, q^{1/2n})$.
Since $q$ was chosen congruent to 3 modulo 4, the 
quadratic subfield of $K_{2n}(q)$ is $\QQ(\sqrt{-q})$ and
it follows that $K_{2n}(q)$ and $\QQ(\zeta_{2n}, q^{1/2n})$ are linearly
disjoint.  By Chebotarev's density theorem, there exists a prime $p$ satisfying all the required constraints
simultaneously, completing the proof when $s = 2$.

Suppose now that $s \ge 3$.

The groups $T_1, \dots , T_s$ in a tame decomposition configuration generate $G$, 
so a realization over $\QQ$ would be a composite of a biquadratic extension $K$ with Galois
group $\Gal(K/\QQ) = \gp{\sigma, \tau}$ and fields
$K_n(p)$, $K_n(q)$, $K_n(l_3) \dots K_n(l_s)$
for some primes $p, q, l_3, \dots , l_s$, each of which is congruent to 1 modulo $n$, whose ramification
information would be given by 
\begin{align} 
G = \gp{\sigma, \tau} \times \gp{x_1} \times \dots \times \gp{x_s} 
     & \isom (\ZZ/2\ZZ \times \ZZ/2\ZZ) \times (\ZZ/n\ZZ)^s  \label{eq:KleinSylowG} \\
T(p) = \gp{\sigma} \times \gp{x_1} \isom \ZZ/(2n) \ZZ,   
   \qquad  & T(q)  = \gp{\tau} \times \gp{x_2} \isom \ZZ/(2n) \ZZ, \label{eq:T1KleinSylowG} \\
T(l_i) = \gp{y_i} \times \gp{x_i}, \qquad & i = 3,4, \dots, s ,  \label{eq:T2KleinSylowG} 
.
\end{align}
where for $i = 3, 4, \dots , s$, the element $y_i$ is one of $1, \sigma, \tau, $ or $\sigma \tau$.  
The biquadratic field $K$ would be uniquely determined by
this ramification information: if
\begin{equation} \label{eq:biquadraticNs}
N_\sigma = \prod_{y_i \in \{\tau, \sigma \tau \} } l_i^* , \qquad
N_\tau = \prod_{y_i \in \{\sigma, \sigma \tau \} } l_i^* , \qquad
N_{\sigma \tau} = \prod_{y_i \in \{\sigma, \tau \} } l_i^* ,
\end{equation}
then $K = \QQ( \sqrt{ N_\sigma } , \sqrt{ N_\tau } )$, with quadratic subfields
$K_\sigma = \QQ( \sqrt{ N_\sigma } )$, $K_\tau = \QQ( \sqrt{ N_\tau } )$,
$K_{\sigma \tau} = \QQ( \sqrt{ N_{\sigma \tau} } )$, fixed by $\sigma$, $\tau$ and $\sigma \tau$, respectively.
We have $p \mid N_\tau, N_{\sigma \tau}$, and $q \mid N_\sigma, N_{\sigma \tau}$.
The primes ramifying in $K_\sigma$ are the primes whose nontrivial inertia group is distinct from
$\gp{\sigma}$, and similarly for the other two quadratic subfields of $K$.  

The remaining configuration information given by a choice of decomposition groups:
\begin{equation} \label{eq:ZconfigKlein}
\begin{aligned}
Z(p) = Z_1 = \gp{\sigma, w_p} \times \gp{x_1, z_1} , \quad  \quad  Z(q)  & = Z_2 = \gp{\tau, w_q} \times \gp{x_2, z_2} , \\
\text{and } \quad Z(l_i)  = Z_i = \gp{y_i, w_i} \times \gp{x_i, z_i} & , \ \  i = 3,4, \dots , s, 
\end{aligned}
\end{equation}
where 
$ z_i = \prod_{j=1}^s x_j^{a_{ij}}$ (and $a_{ii} = 0$) for $ i = 1,2, \dots , s$,
and where $w_p$, $w_q$ and $w_i$ for  $i = 3, 4, \dots , s$ are in $\gp{\sigma, \tau }$.

To prove the existence of suitable primes $p,q,l_3, \dots , l_s$,
we first apply the inductive hypothesis to a decomposition configuration (that will depend on $T_s$ and $Z_s$)
for a group of rank $s-1$ to obtain the primes $p,q,l_3, \dots , l_{s-1}$ and then show the existence
of a prime $l_s$ and a modification of this initial configuration that realizes the rank $s$ configuration.

Suppose first that $T_s = \gp{\sigma}$.
If $T_s = Z_s$, define $\epsilon = 1$ and if $T_s \neq Z_s$, define $\epsilon = \sigma$.
By induction applied to the group $\gp{\sigma, \tau} \times \gp{x_1} \times \dots \times \gp{x_{s-1}}$ of rank $s-1$,
there are primes $p,q,l_3, \dots , l_{s-1}$ and a realization 
$K_0 K_n(p) K_n(q) K_n(l_3) \dots K_n(l_s)$
over $\QQ$ with the tame decomposition configuration 
\begin{align} 
T(p) & = \gp{\sigma} \times \gp{x_1}  &    Z(p) & = \gp{\sigma, w_p} \times \gp{x_1, z_1/x_s^{a_{1,s}} } \label{eq:inductiveTl} \\
T(q) & = \gp{\tau} \times \gp{x_2}  &     Z(q) & = \gp{\tau, \epsilon w_q} \times \gp{x_2, z_2/x_s^{a_{2,s}} } \label{eq:inductiveT2} \\
T(l_i) & = \gp{y_i} \times \gp{x_i}  &     Z(l_i) & = \gp{y_i,w_i} \times \gp{x_i, z_i/x_s^{a_{i,s}} } \, , \ i = 3,4,\dots,s-1
\label{eq:inductiveT3}
\end{align}
obtained from the first $s-1$ conditions in \eqref{eq:ZconfigKlein} by modifying $w_q$ in $Z(q)$ by the $\epsilon$
defined above, and, as 
in the previous subsection, modifying the elements $z_i$ for $1 \le i \le s-1$.
Let $N'_\sigma$, $N'_\tau$ and $N'_{\sigma \tau}$ be the integers in \eqref{eq:biquadraticNs} for the
biquadratic field $K_0$ for this rank $s-1$ realization.

Choose a prime $l_s \equiv 1$ modulo $4n$ such that 

\begin{enumerate}
\item
$ \dbinom{l_s}{q} = 
\begin{cases} 
+1 & \text{if } w_s \in \gp{\sigma} \\
-1 & \text{if } w_s \notin \gp{\sigma} ,
\end{cases}
$

\item
$ \dbinom{l_s}{l} = + 1$, for $l \in \{p, l_3, \dots , l_{s-1} \}$,

\item
the Frobenius for $l_s$ in $K_n(p) K_n(q) K_n(l_3) \cdots K_n (l_{s-1})$
ensures the odd part of the decomposition of $l_s$ in $K_n(p) K_n(q) K_n(l_3) \cdots K_n (l_{s-1})$ is as needed, and

\item
the Frobenius for $l_s$ in $\QQ(\zeta_{n})( p^{1/n}, q^{1/n}, l_3^{1/n}, \dots ,  l_{s-1}^{1/n} )$
ensures (by the Reciprocity Theorem \eqref{thm:reciprocitytheorem})
the odd part of the decomposition of $p,q,l_3, \dots , l_{s-1}$ in $K_n(l_s)$ is as needed.

\end{enumerate}
Again, since the fields involved are linearly disjoint, such a prime $l_s$ exists by Chebotarev's density theorem.

\begin{remark}
The condition in (2) can be relaxed to include only $p$ and those primes $l_i$ whose 
associated $y_i$ in \eqref{eq:inductiveT3} is not trivial.  

\end{remark}

Let $N_\sigma  = N'_\sigma$, $N_\tau = l_s^* N'_\tau = l_s N'_\tau $ and 
$N_{\sigma \tau} =l_s^* N'_{\sigma \tau} =l_s N'_{\sigma \tau}$, and define
$K = \QQ( \sqrt{N_\sigma} , \sqrt{N_\tau} )$.  Then $K$ is a biquadratic extension of $\QQ$ with quadratic subfields
$K_\sigma = \QQ( \sqrt{ N_\sigma } )$, $K_\tau = \QQ( \sqrt{ N_\tau } )$, and 
$K_{\sigma \tau} = \QQ( \sqrt{ N_{\sigma \tau} } )$.  It is then straightforward to check that the
composite field $K K_n(p) K_n(q) K_n(l_3) \dots K_n(l_s)$ is a realization of the decomposition configuration
\eqref{eq:T1KleinSylowG}, \eqref{eq:T2KleinSylowG} and \eqref{eq:ZconfigKlein}.  For example, the even part of
the decomposition behavior (i.e., the 2-primary parts of the ramification and decomposition groups) is 
realized in the biquadratic field $K$, as follows.
The field $K_\sigma$ is the same for $K_0$ and $K$, and the fact that $l_s = l_s^*$ is a square modulo
$p, l_3, \dots , l_{s-1}$ by condition (2) shows the decomposition behavior for these primes in $K$ is as desired.

The decomposition of $l_s$ in $K$ is the question of whether $l_s$ is split or inert in $K_\sigma$ (it is
ramified in the other two quadratic subfields), i.e., by   
the Legendre symbol $\binom{ N_\sigma }{l_s}$, which equals  
$\binom{ l_s }{q}$ by condition (2) and the fact that $l_s \equiv 1$ modulo 4.
By condition (1), this is precisely the desired decomposition for $l_s$ in \eqref{eq:T2KleinSylowG}, as desired.

The decomposition of $q$ in $K$ is determined by the splitting of $q$ in 
$K_\tau$, so by the Legendre symbol $\binom{ N_\tau l_s }{q} = \binom{ N_\tau}{q} \binom{l_s }{q}$.  
By the choice of the inductive configuration, $\binom{ N_\tau}{q} = +1$ if and only if $\epsilon w_q \in \gp{\tau}$
and by condition (2), $\binom{ l_s }{q} = +1$ if and only if $w_s \in \gp{\sigma}$, i.e., if and only if
$T_s = Z_s$.  By the definition of $\epsilon$, a quick check shows
$\binom{ N_\tau l_s }{q} = +1$ if and only if
$w_q \in \gp{\tau}$, so $q$ decomposes as desired in $K$.

The cases $T_s = \gp{\tau}$ and $T_s = \gp{\sigma \tau}$ are handled similarly, as follows.
In both cases, let $\epsilon = 1$ if $Z_s = T_s$, let $\epsilon$ be the generator of $T_s$ if $Z_s \ne T_s$,
and for the inductive rank $s-1$ decomposition configuration
replace \eqref{eq:inductiveTl} and \eqref{eq:inductiveT2} with 
\begin{align} 
T(p) & = \gp{\sigma} \times \gp{x_1}  &    Z(p) & = \gp{\sigma, \epsilon w_p} \times \gp{x_1, z_1/x_s^{a_{1,s}} }  
\tag{ {\ref{eq:inductiveTl}}' }  \\
T(q) & = \gp{\tau} \times \gp{x_2}  &     Z(q) & = \gp{\tau, w_q} \times \gp{x_2, z_2/x_s^{a_{2,s}} } 
\tag{ {\ref{eq:inductiveT2}}' }  
\end{align}
Let $N'_\sigma$, $N'_\tau$ and $N'_{\sigma \tau}$ be the integers in \eqref{eq:biquadraticNs} for the
biquadratic field $K_0$ for this rank $s-1$ realization
and if $T_s = \gp{\tau}$ (resp., $T_s = \gp{\sigma \tau}$), let 
$N_\sigma  = l_s N'_\sigma $, $N_\tau = N'_\tau $ and 
$N_{\sigma \tau} = l_s N'_{\sigma \tau}$
(resp., $N_\sigma  = l_s N'_\sigma $, $N_\tau = l_s  N'_\tau $ and 
$N_{\sigma \tau} = N'_{\sigma \tau}$).  Set $K = \QQ( \sqrt{N_\sigma}, \sqrt{N_\tau})$.
Choose (by Chebotarev's density theorem) a prime $l_s \equiv 1$ modulo $4n$ so that 
(1) $\binom{l_s}{p} = +1$ if $T_s = Z_s$ and $ \binom{l_s}{p} = -1$ if $T_s \ne Z_s$, (2)
$\binom{l_s}{l} = +1$ for $l \in \{q, l_3, \dots , l_{s-1} \}$, and so that the earlier conditions (3) and (4) 
are all satisfied.  

Finally, if $T_s = 1$, set $\epsilon = 1$, take 
\eqref{eq:inductiveTl}-\eqref{eq:inductiveT3} for the inductive configuration, take $K = K_0$,
and replace conditions (1) and (2) for the prime $l_s \equiv 1$ modulo $4n$ by the condition that its Frobenius in $K$ gives
the correct decomposition group $Z_s$.

In all cases, it is easy to check as before that 
$K K_n(p) K_n(q) K_n(l_3) \dots K_n(l_s)$ is a realization of the decomposition configuration, 
completing the proof of Theorem  \eqref{thm:Sylow2abelian}.

\section{Groups of Small Order}

In this section we consider the realizations of (minimal) tame decomposition configurations of some nonabelian groups. 
In many cases there are restrictions on the configurations that can be realized (for example, every minimally ramified 
tame $Q_8$ extension is necessarily a split tame decomposition configuration), and, 
for those that can be realized, we give explicit realizations.  We also observe that it is a 
difficult (open) question to determine if a given tame decomposition 
configuration that can be realized over $\QQ$ in fact has infinitely many different realizations; this 
is already an interesting problem for the groups considered here.

\subsection{Nonabelian Groups of order 8}\label{order8}

Consider first $G=Q_8$, the quaternion group of order $8$.  Then $s=2$ and a tame decomposition configuration  $(G, \mathcal T, \mathcal Z)$ 
must have $T_1$ and $T_2$ cyclic groups of order $4$;  up to evident equivalence, there are three possible
configurations, depending on whether $T_i = Z_i$ or not.  If $K/\Q$ is a realization 
then the unique biquadratic subfield $F$ of $K$, $F= \Q(\sqrt {p_1^*},\sqrt{p_2^*}),$ must be totally real,
hence $p_1 \equiv p_2 \equiv 1 \pmod 4$.  As before, this imposes a quadratic reciprocity constraint that
implies $Z_1 =T_1$ if and only if $Z_2=T_2$.   But $T_1 \subsetneq Z_1$ means that $Z_1=Q_8$, and since only 
$\ZZ/8 \ZZ$ and $\ZZ/2\ZZ \times \ZZ/4\ZZ$ occur as Galois groups of order 8 over 
$\Q_p$ if  $p \equiv 1 \pmod 4$ this cannot occur. 
Alternatively, a theorem of Witt \cite{W} states that 
$\QQ(\sqrt a, \sqrt b)$ can be embedded in a $Q_8$ extension if and only if $(-a,-b) = (-1,-1)$ (Hilbert symbols), which for
$a = p_1^*$, $b = p_2^*$ is equivalent to $p_1 \equiv p_2 \equiv 1 \pmod 4$ and 
$\big(\frac{p_1}{p_2}\big)= \big(\frac{p_2}{p_1}\big) = +1$.
Therefore only the configuration with $Z_1=T_1$ and $Z_2=T_2$ could be realizable. In 
\cite{F}, Fr\"ohlich shows this Witt condition is both necessary and sufficient for the existence of a (unique)
$Q_8$ extension of $\QQ$ ramified only at the primes $p_1$ and $p_2$, and 
Schmid in \cite{S} identifies it explicitly.

Next consider $G=D_8 = \langle r,s \mid r^4 = s^2 = 1, sr = r^{-1} s \rangle$, the dihedral group of order $8$.  
Again the rank is 2, 
and up to an isomorphism of $G$ or an interchange of $\{ T_1, Z_1 \}$ and $\{ T_2, Z_2 \}$, there are seven
possible tame decomposition configurations for $G$.  Searching the number field data base \cite{JR} we find
that all seven configurations for $D_8$ can be realized over $\QQ$; Table \ref{table:D8} gives the possible tame configurations 
$\{ T_1, Z_1 \}$ and $\{ T_2, Z_2 \}$ for $D_8$, and for each a polynomial and primes $p_1$ and $p_2$
(with ramification index $e$ and inertial degree $f$) realizing the configuration.

\begin {table}[ht]    
\begin{center}
  \begin{tabular}{c  c  c  c | c  c | p{5.3cm} }
$T_1$ & $Z_1$ & $T_2$ & $Z_2$  & $p_1 \ (e,f)$ & $p_2 \ (e,f)$ & \quad \qquad \text{realization} \\ \hline
\multicolumn{1}{|r}{$\gp{r}$} & $\gp{r}$    &   $\gp{s}$  &   $\gp{s}$  & 5 (4,1) & 29 (2,1)
             & \multicolumn{1}{p{5.3cm}|}{$x^8 - x^6 - 4 x^4 - 16 x^2 + 256$} \\ \hline
\multicolumn{1}{|r}{$\gp{r}$} & $\gp{r}$    &   $\gp{s}$  &   $\gp{s,r^2}$   & 13 (4,1)   & 3 (2,2)
             & \multicolumn{1}{p{5.3cm}|}{$x^8 - 9 x^6 + 32 x^4 - 9 x^2 + 1$} 
             \\ \hline
\multicolumn{1}{|r}{$\gp{r}$} & $G$    &   $\gp{s}$  &   $\gp{s}$   &  23 (4,2)  &  3 (2,1)
             & \multicolumn{1}{p{5.3cm}|}{$x^8 - 3 x^7 + 7 x^6 - 12 x^5 - 8 x^4 + 84 x^3 + 159 x^2 + 63 x + 9$} 
             \\ \hline
\multicolumn{1}{|r}{$\gp{r}$} & $G$    &   $\gp{s}$  &   $\gp{s,r^2}$   & 3 (4,2)  &  7 (2,2)
             & \multicolumn{1}{p{5.3cm}|}{$x^8 - 3 x^7 + 4 x^6 - 3 x^5 + 3x^4 - 3x^3 + 4x^2 - 3x + 1$} 
             \\ \hline
\multicolumn{1}{|r}{$\gp{s}$} & $\gp{s}$   &   $\gp{sr}$  &   $\gp{sr}$   & 3 (2,1)   &  37 (2,2)
             & \multicolumn{1}{p{5.3cm}|}{$x^8 - 5x^6 + 28x^4 + 15x^2 + 9$} 
             \\ \hline
\multicolumn{1}{|r}{$\gp{s}$} & $\gp{s}$    &   $\gp{sr}$  &   $\gp{sr,r^2}$  & 5 (2,1)  & 41 (2,2)
             & \multicolumn{1}{p{5.3cm}|}{$x^8 + 15x^6 + 48x^4 + 15x^2 + 1$}
             \\ \hline
\multicolumn{1}{|r}{$\gp{s}$} & $\gp{s,r^2}$    &   $\gp{sr}$  &   $\gp{sr,r^2}$ & 3 (2,2) & 13 (2,2)
             & \multicolumn{1}{p{5.3cm}|}{$x^8 - x^7 + 2x^6 + 3x^5 - x^4 + 3x^3 + 2x^2 - x + 1$}
             \\ \hline    
  \end{tabular}
\\[5pt] 
 \caption {Decomposition Configurations for $D_8$} \label{table:D8} 
 \end{center}
\end {table}

\subsection{The groups $D_{10}$, $A_4$, $F_{20}$} \label{d10f20a4}
The dihedral group of order 10, the alternating group of order 12, and the Frobenius group of order 20 are all of rank 
$s = 1$ and have, up to
an isomorphism of the group, a unique tame decomposition configuration.  Each is realized over $\QQ$:

$D_{10}: x^5 - 2 x^4 + 2 x^3 - x^2 + 1 \ (p = 47)$,

$A_4 : x^4 - x^3 - 7 x^2 + 2 x + 9  \ (p = 163)$,

$F_{20}: x^5 - 2 x^4 + 7 x^3 - 4 x^2 + 11 x + 6 \ (p = 101)$.

\begin{remark}
It is not difficult to show that any Frobenius group $F = K \rtimes H$ whose Frobenius complement $H$ is
cyclic is of rank $s = 1$ and has a unique tame decomposition configuration $T = Z = H$ up to isomorphism.
This includes the dihedral groups $D_{2n}$ of order $2n$ where $n$ is odd ($H$ any subgroup of order 2) 
as well as $A_4$ ($H$ any subgroup of order 3) and $F_{20}$ (where $H$ is a cyclic subgroup of order 4).  
For such groups $F$ the existence of a realization over $\QQ$ is just the question of realizing $F$ as
a Galois group over $\QQ$ by an extension with a single (tamely) ramified prime.  

\end{remark}

\begin{remark}
In any realization $K/\Q$ of the unique tame decomposition configuration for $D_{2n}$ with $n$ odd (for example, 
for the symmetric group $S_3$), the (unique) quadratic subfield would be
$k = \Q(\sqrt{(-1)^{(p-1)/2} p})$ for some odd prime $p\in \Z$ with $K/k$ unramified at finite primes.  
It follows by class field theory that there is a tame realization for $D_{2n}$ if and only if 
the class group of the quadratic field $\Q(\sqrt{(-1)^{(p-1)/2} p})$ contains an element of order $n$ for the 
odd prime $p$ not dividing $n$.  For $S_3$, this occurs (for $p = 23$ for example), presumably infinitely often, 
with interesting questions on the statistics.

\end{remark}

\subsection{The groups $S_4$, $A_5$, $S_5$} \label{a5s5}

Up to isomorphism, there are four distinct tame decomposition configurations for the 
symmetric group $S_4$, six configurations for the alternating group $A_5$,
and seven configurations for the symmetric group $S_5$.  
Searching the number field data base \cite{JR} we find
that all seventeen configurations for these three groups of rank $s = 1$
can be realized over $\QQ$; Tables \ref{table:S4},  \ref{table:A5} and \ref{table:S5}
give the possible tame configurations 
$\{ T, Z \}$ and for each a prime $p$ and a polynomial $f(x)$ whose Galois closure realizes the configuration. 
We also indicate the splitting of the prime $p$ in the extension $F$ generated by a root of $f(x)$, which
can be computed from the double coset decomposition of the group $G$ under the actions of $Z$ and the
subgroup $H$ of $K$ fixing $F$.  The cases where distinct $Z$ yield the same splitting in $F$ are 
distinguished by the largest subgroup $N$ of $Z \cap H$ normal in $Z$---the quotient $Z/N$ is provided
in the splitting data in \cite{JR} as the Galois closure of the appropriate completion of $F$.  
All primes have residue degree 1 unless otherwise indicated. 

\begin {table}[ht]    
\begin{center}
  \begin{tabular}{c  c  | c  |  c | p{4.5cm} }
$T$ & $Z$ &  $p$ & \text{splitting} & \quad \qquad \text{realization} \\ \hline
\multicolumn{1}{|c}{$\gp{ (1 \ 2)}$} &  
             \multicolumn{1}{p{3.0cm} | }{ $\gp{ (1 \ 2) }$ }  
             &   283
             & $ \ \underset{ }{\wp_1^2}  \ \underset{ }{\wp_2} \ \underset{ }{\wp_3} $
             & \multicolumn{1}{p{4.5cm}|}{$x^4 - x - 1$} 
             \\ \hline
\multicolumn{1}{ | c}{$\gp{(1 \ 2) }$} &  
             \multicolumn{1}{p{3.0cm} | }{  $\gp { (1 \ 2),  (3 \ 4) } $ }
             &    229  
             & $ \ \underset{ }{\wp_1^2}  \ \underset{(f=2)}{\wp_2}$
             & \multicolumn{1}{p{4.5cm}|}{$x^4 - x + 1$} 
             \\ \hline
\multicolumn{1}{|c}{$\gp{(1 \ 2 \ 3 \ 4)}$} & 
             \multicolumn{1}{p{3.0cm} | }{$\gp{(1 \ 2 \ 3 \ 4) }$}  
             &   229
             & $\ \underset{ }{\wp^4}$ 
             & \multicolumn{1}{p{4.5cm}|}{$ x^4 - x^3 + 29 x^2 - 43 x + 17$} 
             \\ \hline
\multicolumn{1}{|c}{$\gp{(1 \ 2 \ 3 \ 4) }$} & 
             \multicolumn{1}{p{3.0cm} | }{  $\gp{ (1 \ 2 \ 3 \ 4), (1 \ 3) }$ }    
             &    59
             & $\ \underset{ }{\wp^4}$
             & \multicolumn{1}{p{4.5cm}|}{$x^4 - x^3 - 7 x^2 + 11 x + 3$} 
             \\ \hline
  \end{tabular}
\\[5pt] % You can adjust how far below the table the text should appear
 \caption {Decomposition Configurations for $S_4$} \label{table:S4} 
 \end{center}
\end {table}

\begin {table}[ht] 
\begin{center}
  \begin{tabular}{c  c  | c  |  c | p{4.5cm} }
$T$ & $Z$ &  $p$ & \text{splitting} & \quad \qquad \text{realization} \\ \hline
\multicolumn{1}{|c}{$\gp{ (1 \ 2 \ 3)}$} &  
             \multicolumn{1}{p{2.2cm} | }{ $\gp{ (1 \ 2 \ 3) }$ }  
             &    10267   
             & $ \ \underset{}{\wp_1^3}  \ \underset{}{\wp_2} \ \underset{}{\wp_3} $
             & \multicolumn{1}{p{4.5cm}|}{$x^5 - 25 x^3 - 7 x^2 + 116 x - 45$} 
             \\ \hline
\multicolumn{1}{ | c}{$\gp{(1 \ 2 \ 3) }$} &  
             \multicolumn{1}{p{2.2cm} | }{  $\langle (1 \ 2 \ 3)$, \hfill\break ${ } \ \ (1 \ 2 )(4 \ 5) \rangle$ }
             &   4253
             & $\ \underset{}{\wp_1^3}  \ \underset{(f=2)}{\wp_2}$
             & \multicolumn{1}{p{4.5cm}|}{$x^5 - 2 x^4 - 10 x^3 + 23 x^2 - 6 x - 4$} 
             \\ \hline
\multicolumn{1}{|c}{$\gp{(1 \ 2 \ 3 \ 4 \ 5)}$} & 
             \multicolumn{1}{p{2.2cm} | }{$\gp{(1 \ 2 \ 3 \ 4 \ 5) }$}  
             &   1951
             & $\ \underset{}{\wp_1^5}$ 
             & \multicolumn{1}{p{4.5cm}|}{$ x^5  - x^4  - 780 x^3 +9911 x^2 - 24208 x + 15952$} 
             \\ \hline
\multicolumn{1}{|c}{$\gp{(1 \ 2 \ 3 \ 4 \ 5) }$} & 
             \multicolumn{1}{p{2.2cm} | }{  $\langle (1 \ 2 \ 3 \ 4 \ 5)$, ${ } \ \  (2 \ 5) (3 \ 4) \rangle$ }    
             &    1039
             & $\ \underset{}{\wp_1^5}$
             & \multicolumn{1}{p{4.5cm}|}{$x^5 - 2 x^4 - 414 x^3 + 4945 x^2 - 16574 x + 5191$} 
             \\ \hline
\multicolumn{1}{|c}{$\gp{(1 \ 2 ) (3 \ 4 ) }$} & 
              \multicolumn{1}{p{2.2cm} | }{$\gp{(1 \ 2 ) (3 \ 4 )}$}  
             &    2083   
             & $ \ \underset{}{\wp_1^2} \ \underset{}{\wp_2^2} \ \underset{}{\wp_3} $
             & \multicolumn{1}{p{4.5cm}|}{$x^5 - x^4 + 5 x^3 + 11 x^2 + 4 x - 1$} 
             \\ \hline
\multicolumn{1}{|c}{$\gp{(1 \ 2 ) (3 \ 4 )}$} & 
              \multicolumn{1}{p{2.2cm} | }{  $\langle (1 \ 2 ) (3 \ 4 )$,  ${ } \ \  (1 \ 3 ) (2 \ 4 ) \rangle$ } 
             &   653
             & $\ \underset{(f=2)}{\wp_1^2} \ \underset{}{\wp_2}$
             & \multicolumn{1}{p{4.5cm}|}{$x^5 + 3 x^3 - 6 x^2 + 2x - 1$} 
             \\ \hline
  \end{tabular}
\\[5pt] % You can adjust how far below the table the text should appear
 \caption {Decomposition Configurations for $A_5$} \label{table:A5} 
 \end{center}
\end {table}

\begin {table}[ht] 
\begin{center}
  \begin{tabular}{c  c  | c  |  c | p{3.5cm} }
$T$ & $Z$ &  $p$ & \text{splitting} & \quad \ \ \text{realization} \\ \hline
\multicolumn{1}{|c}{ $\gp{ (1 \ 2) } $ } &  
             \multicolumn{1}{p{2.3cm} | }{ $ \gp{(1 \ 2) } $ }  
             &    13219
             &  $ \underset{}{\wp_1^2} \ \underset{}{\wp_2} \ \underset{}{\wp_3} \ \underset{}{\wp_4} $
             & \multicolumn{1}{p{3.5cm}|}{$  x^5 - 2 x^2 - x + 1 $} 
             \\ \hline
\multicolumn{1}{|c}{ $\gp{ (1 \ 2) } $ } &  
             \multicolumn{1}{p{2.3cm} | }{ $ \gp{ (1 \ 2), (3 \ 4) } $ }  
             &    1609
             &  $ \underset{}{\wp_1^2} \ \underset{(f=2)}{\wp_2} \ \underset{}{\wp_3}{} $
             & \multicolumn{1}{p{3.5cm}|}{$ x^5 - x^3 - x^2 + x + 1  $} 
             \\ \hline
\multicolumn{1}{|c}{ $\gp{ (1 \ 2)  } $ } &  
             \multicolumn{1}{p{2.3cm} | }{  $\langle (1 \ 2), \ \ \ \ \ $  ${ } \ \ (3 \ 4 \ 5) \rangle$ } 
             &    4903
             &  $ \underset{}{\wp_1^2} \ \underset{(f=3)}{\wp_2} $
             & \multicolumn{1}{p{3.5cm}|}{$ x^5 - x^4 - x^3 + 2 x^2 - x - 1 $} 
             \\ \hline
\multicolumn{1}{|c}{ $\gp{  (1 \ 2) (3 \ 4 \ 5) } $ } &  
             \multicolumn{1}{p{2.3cm} | }{ $ \gp{ (1 \ 2) (3 \ 4 \ 5)  } $ }  
             &   151
             &  $ \underset{}{\wp_1^2} \ \underset{}{\wp_2^3} $
             & \multicolumn{1}{p{3.5cm}|}{$ x^5 - 2 x^4 - x^3 + 7 x^2 - 13 x + 7 $} 
             \\ \hline
\multicolumn{1}{|c}{ $\gp{ (1 \ 2) (3 \ 4 \ 5) } $ } &  
             \multicolumn{1}{p{2.3cm} | }{  $\langle (1 \ 2) (3 \ 4 \ 5)$,  ${ } \ \ (3 \ 4) \rangle$ }
             &    101   
             &  $ \underset{}{\wp_1^2} \ \underset{}{\wp_2^3} $
             & \multicolumn{1}{p{3.5cm}|}{$ x^5 - x^4 - 6 x^3 + x^2 + 18 x - 4 $} 
             \\ \hline
\multicolumn{1}{|c}{ $\gp{ (1 \ 2 \ 3 \ 4 ) } $ } &  
             \multicolumn{1}{p{2.3cm} | }{ $ \gp{  (1 \ 2 \ 3 \ 4 )  } $ }  
             &    269   
             &  $ \underset{}{\wp_1^4} \ \underset{}{\wp_2} $
             & \multicolumn{1}{p{3.5cm}|}{$ x^5 - x^4 - 15 x^3 - 11 x^2 + 11 x - 10 $} 
             \\ \hline
\multicolumn{1}{|c}{ $\gp{ (1 \ 2 \ 3 \ 4 ) } $ } &  
             \multicolumn{1}{p{2.3cm} | }{ $\langle (1 \ 2 \ 3 \ 4 ), \ \ \ $  ${ } \ \ (1 \ 3) \rangle$ } 
             &    619   
             &  $ \underset{}{\wp_1^4} \ \underset{}{\wp_2} $
             & \multicolumn{1}{p{3.5cm}|}{$ x^5 - x^4 - 13 x^3 - 6 x^2 + 8 x + 47 $} 
             \\ \hline
  \end{tabular}
\\[5pt] % You can adjust how far below the table the text should appear
 \caption {Decomposition Configurations for $S_5$} \label{table:S5} 
 \end{center}
\end {table}

\begin{remark}
We note that a realization $K$ for the tame decomposition configuration $(S_n, \gp{(1 \, 2)},\gp{(1 \, 2)})$ gives a Galois
extension of the quadratic field $k = Q(\sqrt{(-1)^{(p-1)/2} p}$ with $\Gal(K/k) \isom A_n$ that is unramified over
$k$ at all finite primes and in which the prime over $p$ in $k$ splits completely.  
The existence of quadratic fields $k$ with class number 1 but which nevertheless have nonabelian Galois extensions
unramified outside finite primes was noted
by Artin, who gave the example $k = \QQ(\sqrt{19 \cdot 151})$.  The first realization
of the tame decomposition configuration $(S_5, \gp{(1 \, 2)},\gp{(1 \, 2)})$ 
where the class number of $k$ is 1 occurs for the splitting field of $ x^5 - 2 x^4 - 3 x^3 + 5 x^2 + x - 1$ where
$p = 36497$, and in this case $K$ is totally real, so $K/k$ is also unramified at the infinite primes. 

\end{remark}

\subsection{The group $PSL(2,7)$} \label{psl27}

The simple group $G = PSL(2,7) = GL_3(\FF_2)$ of order 168 has the 
presentation $G = \gp{a,b \mid a^2 = b^3 = (a b)^7 = [a,b]^4 = 1   }$ and an embedding in $S_7$ in which
$a =  (1 \ 2)(3 \ 6)$, $b = (2 \ 6 \ 7)(3 \ 4 \ 5)$, $a b = (1 \ 2 \ 3 \ 4 \ 5 \ 6 \ 7)$. The nonidentity elements in $G$
have orders 2,3,4 and 7. Let
$r = (1 \ 3 \ 2 \ 6)(5 \ 7) \in G$ and $s = (1 \ 2)(5 \ 7) \in G$.  Then $r^2 = a$ and $\gp{r,s} \isom D_8$,
the unique Sylow-2 subgroup containing $\gp{a}$ as a normal subgroup, is the normalizer in $G$ of both
$\gp{a}$ and $\gp{r}$. The normalizer in $G$ of $\gp{b}$ is $\gp{b, u} \isom S_3$ where $u = (2 \ 6)(4 \ 5)$ , and
the normalizer in $G$ of $a b$ is the Frobenius group of order 21, $\gp{ a b , v} $, where $v = (1 \ 2 \ 4)(3 \ 6 \ 5) $.
Up to isomorphism,
there are nine distinct tame decomposition configurations for $G$.  
In this case, the number field data base \cite{JR} provides realizations for the five configurations whose
inertia groups have order 2 or 4.  Realizations for the other possible inertia groups would involve
totally real fields (see \cite{JR2}), and there are substantially fewer of these available: of the approximately 100
totally real septic fields currently in \cite{JR} and \cite{LMFDB}, and the 138 currently in \cite{KM},
none are minimally tamely ramified at a single prime
with odd order inertia group.  The results for $PSL(2,7)$ are summarized in Table \ref{table:PSL27}, where
all primes have residue degree 1 unless otherwise indicated.

We note that for inertia group of order 3, there are examples in the databases that are tamely ramified at
two primes (e.g., at 5 with inertia of order 3, and at 6247 with inertia of order 2).  For
inertia group of order 7, the smallest number of ramified primes in the databases is three, with a single tamely
ramified example (ramified at 11 with inertia of order 7, at 5 with inertia of order 2, 
and at 19 with inertia of order 4).  

\begin {table}[ht] 
\begin{center}
  \begin{tabular}{c  c  | c  |  c | p{6.5cm} }
$T$ & $Z$ &  $p$ & \text{splitting} & \quad \qquad \text{realization} \\ \hline
\multicolumn{1}{|c}{$\gp{ a }$} &  
             \multicolumn{1}{p{1.2cm} | }{ $\gp{ a }$ }  
             &   6971            
             & $\underset{ }{\wp_1^2}  \ \underset{ }{\wp_2^2} \ \underset{ }{\wp_3} \ \underset{ }{\wp_4} \ \underset{ }{\wp_5}$
             & \multicolumn{1}{p{6.5cm}|}{$x^7 - 3 x^6 + 4 x^5 - 3 x^3 + 4 x^2 - x - 1$} 
             \\ \hline
\multicolumn{1}{|c}{$\gp{ a }$} &  
             \multicolumn{1}{p{1.2cm} | }{ $\gp{ r }$ }  
             &   3803
             & $\underset{(f=2)}{\wp_1^2}  \ \underset{(f=2)}{\wp_2} \ \underset{ }{\wp_3} $
             & \multicolumn{1}{p{6.5cm}|}{$x^7 - 2 x^5 - 4 x^4 - x^3 + 5 x^2 + 5 x + 1$} 
             \\ \hline
\multicolumn{1}{|c}{$\gp{ a }$} &  
             \multicolumn{1}{p{1.2cm} | }{ $\gp{ a, s }$ }  
             &   2741
             & $\underset{ }{\wp_1^2}  \ \underset{ }{\wp_2^2} \ \underset{(f=2)}{\wp_3} \ \underset{ }{\wp_4}$
             & \multicolumn{1}{p{6.5cm}|}{$x^7 - x^6 + 2 x^5 - 3 x^4 - 4 x^3 + 2 x^2 + 3 x + 1$} 
             \\ \hline
%
%
%   following is the sibling field decomposition
%
%     \multicolumn{1}{|c}{$\gp{ a }$} &  
%                  \multicolumn{1}{p{1.2cm} | }{ $\gp{ a, s r^3 }$ }  
%                  &   2741
%                  & $\underset{(f=2)}{\wp_1^2}  \ \underset{ }{\wp_2} \ \underset{ }{\wp_3} \ \underset{ }{\wp_4}$
%                  & \multicolumn{1}{p{6.5cm}|}{$x^7 - x^6 - 4 x^5 + x^4 + 4 x^3 - 3 x + 1$} 
%                  \\ \hline
%     
%
\multicolumn{1}{|c}{$\gp{ b }$} &  
             \multicolumn{1}{p{1.2cm} | }{ $\gp{ b }$ }  
             &    
             & $\underset{ }{\wp_1^3}  \ \underset{ }{\wp_2^3} \ \underset{ }{\wp_3} $
             & \multicolumn{1}{p{6.5cm}|}{unknown } 
             \\ \hline
\multicolumn{1}{|c}{$\gp{ b }$} &  
             \multicolumn{1}{p{1.2cm} | }{ $\gp{ b, u }$ }  
             &    
             & $\underset{ }{\wp_1^3}  \ \underset{ }{\wp_2^3} \ \underset{ }{\wp_3} $
             & \multicolumn{1}{p{6.5cm}|}{unknown } 
             \\ \hline
\multicolumn{1}{|c}{$\gp{ r }$} &  
             \multicolumn{1}{p{1.2cm} | }{ $\gp{ r }$ }  
             &   373
             &  $\underset{ }{\wp_1^4}  \ \underset{ }{\wp_2^2} \ \underset{ }{\wp_3} $
             & \multicolumn{1}{p{6.5cm}|}{$x^7 - x^6 - 10 x^5 + 14 x^4 + x^3 - 4 x^2 - 3 x + 5$} 
             \\ \hline
\multicolumn{1}{|c}{$\gp{ r }$} &  
             \multicolumn{1}{p{1.2cm} | }{ $\gp{ r,s }$ }  
             &   227
             & $\underset{ }{\wp_1^4}  \ \underset{ }{\wp_2^2} \ \underset{ }{\wp_3} $
             & \multicolumn{1}{p{6.5cm}|}{ $ x^7 + 2 x^5 - 4 x^4 - 5 x^3 - 4 x^2 - 3 x + 10 $ } 
             \\ \hline
\multicolumn{1}{|c}{$\gp{ a b }$} &  
             \multicolumn{1}{p{1.2cm} | }{ $\gp{ a b }$ }  
             &    
             & $\underset{ }{\wp_1^7}$ 
             & \multicolumn{1}{p{6.5cm}|}{ unknown } 
             \\ \hline
\multicolumn{1}{|c}{$\gp{ a b }$} &  
             \multicolumn{1}{p{1.2cm} | }{ $\gp{ a b , v} $ }
             &    
             & $\underset{ }{\wp_1^7}$ 
             & \multicolumn{1}{p{6.5cm}|}{ unknown } 
             \\ \hline
  \end{tabular}
\\[5pt] % You can adjust how far below the table the text should appear
 \caption {Decomposition Configurations for $G = PSL(2,7) = GL_3(\FF_2)$} \label{table:PSL27} 
 \end{center}
\end {table}

\begin{remark}
The configurations $(G, \gp{ a }, \gp{ a, s } )$ and $(G, \gp{ a }, \gp{ a, s r^3 } )$ are isomorphic under the
outer automorphism of $G$ that maps $a$ to $a$ and $b$ to $b^{-1}$ (but are not isomorphic under an inner
automorphism), hence define the same tame decomposition configuration
for $G$.  We note that this configuration manifests itself in the two nonisomorphic (`sibling' and
arithmetically equivalent)
subfields of degree 7 in any realization as different splittings of the prime $p$ (making them appear at first to be distinct
decomposition configurations in the \cite{JR} database).  For example, the sibling septic to that in the table above for
the configuration $(G, \gp{ a }, \gp{ a, s })$ 
is the field defined by $f = x^7 - x^6 - 4 x^5 + x^4 + 4 x^3 - 3 x + 1$, in which the prime $p = 2741$ splits as
$\underset{(f=2)}{\wp_1^2}  \ \underset{ }{\wp_2} \ \underset{ }{\wp_3} \ \underset{ }{\wp_4}$.

\end{remark}

\section{Conclusion}

For finite abelian Galois groups, the only obstruction to obtaining a realization for a (minimal) tame 
decomposition configuration arises from constraints imposed by quadratic reciprocity and
for nonabelian groups of {\it odd} order we have no examples of minimal tame decomposition configurations 
that cannot be realized.  Our computations suggest, for example, that every tame decomposition 
configuration can be realized for groups normally generated by a single element, so in particular for
every finite simple group and for the symmetric groups $S_n$.   

\section{Acknowledgments}
We would like to thank Richard Foote for many helpful conversations.

\end{document}